\documentclass[a4paper, 11pt]{article}

\vspace{-1cm}

\title{\Large\bf
Identities of inverse Chevalley type for \\ graded characters of level-zero Demazure submodules over \\ quantum affine algebras of type $C$%
} 
\author{%
Takafumi Kouno \\ 
 \small Department of Mathematics, Faculty of Science and Engineering, Waseda University, \\ 
 \small 3-4-1 Okubo, Shinjuku-ku, Tokyo 169-8555, Japan \\ 
 \small (e-mail: {\tt t.kouno@kurenai.waseda.jp}) \\[5mm]
Satoshi Naito \\
 \small Department of Mathematics, Tokyo Institute of Technology, \\ 
 \small 2-12-1 Oh-Okayama, Meguro-ku, Tokyo 152-8551, Japan \\ 
 \small (e-mail: {\tt naito@math.titech.ac.jp}) \\[5mm]
Daniel Orr \\ 
 \small Department of Mathematics (MC 0123), 460 McBryde Hall, Virginia Tech, \\ 
 \small 225 Stanger St., Blacksburg, VA 24061, U.S.A. \\ 
 \small (e-mail: {\tt dorr@vt.edu})
}
\date{}

\usepackage[truedimen,margin=23truemm]{geometry}
\usepackage{hyperref}
\usepackage{amsmath, amssymb}
\usepackage{amsthm}
\usepackage{bm}
\usepackage{enumerate}

\theoremstyle{plain}
\newtheorem{thm}{Theorem}[section]
\newtheorem{lem}[thm]{Lemma}
\newtheorem{prop}[thm]{Proposition}
\newtheorem{cor}[thm]{Corollary}
\newtheorem{conj}[thm]{Conjecture}

\newtheorem{ithm}{Theorem}

\theoremstyle{definition}
\newtheorem{defn}[thm]{Definition}
\newtheorem{exm}[thm]{Example}

\theoremstyle{remark}
\newtheorem{rem}[thm]{Remark}

\numberwithin{equation}{section}

\newcommand{\BZ}{\mathbb{Z}}
\newcommand{\BR}{\mathbb{R}}
\newcommand{\BC}{\mathbb{C}}

\newcommand{\CA}{\mathcal{A}}
\newcommand{\CO}{\mathcal{O}}
\newcommand{\CP}{\mathcal{P}}
\newcommand{\CS}{\mathcal{S}}

\newcommand{\bp}{\mathbf{p}}

\newcommand{\bi}{\bm{i}}

\newcommand{\bQG}{\mathbf{Q}_{G}}
\newcommand{\rat}{\mathrm{rat}}
\newcommand{\bQGr}{\bQG^{\rat}}

\newcommand{\q}{\mathsf{q}}

\newcommand{\Fg}{\mathfrak{g}}
\newcommand{\Fh}{\mathfrak{h}}

\newcommand{\ve}{\varepsilon}
\newcommand{\vpi}{\varpi}

\newcommand{\bchi}{\bm{\chi}}

\newcommand{\af}{\mathrm{af}}

\newcommand{\Par}{\mathrm{Par}}
\newcommand{\bPar}{\overline{\Par}}

\newcommand{\pair}[2]{\langle #1, #2 \rangle}
\newcommand{\bigpair}[2]{\left\langle #1, #2 \right\rangle}
\newcommand{\bra}[1]{[\![ #1 ]\!]}
\newcommand{\pra}[1]{(\!( #1 )\!)}

\newcommand{\vtl}{\vartriangleleft}

\DeclareMathOperator{\ed}{end}
\DeclareMathOperator{\down}{down}
\DeclareMathOperator{\wt}{wt}
\DeclareMathOperator{\height}{height}
\DeclareMathOperator{\sgn}{sgn}
\DeclareMathOperator{\gch}{gch}
\DeclareMathOperator{\QBG}{QBG}
\DeclareMathOperator{\Funorg}{Fun}

\newcommand{\Fun}{\Funorg_{P}(\BC\pra{q^{-1}}[P])}
\newcommand{\Funneg}{\Funorg_{P}^{\mathrm{neg}}(\BC\pra{q^{-1}}[P])}
\newcommand{\Funess}{\Funorg_{P}^{\mathrm{ess}}(\BC\pra{q^{-1}}[P])}

\makeatletter
\renewcommand\section{\@startsection{section}{1}{0pt}
{-3.5ex plus -1ex minus -.2ex}{1.0ex plus .2ex}{\large\bf}}
\renewcommand\subsection{\@startsection{subsection}{1}{0pt}
{2.5ex plus 1ex minus .2ex}{-1em}{\bf}}
\makeatother

\newenvironment{enu}%
{\begin{enumerate}[\upshape {(}1{)}]}%
{\end{enumerate}}

\begin{document}
\maketitle

\begin{abstract}
We provide identities of inverse Chevalley type for the graded characters of level-zero Demazure submodules of extremal weight modules over a quantum affine algebra of type $C$.
These identities express the product $e^{\mu} \gch V_{x}^{-}(\lambda)$ of the (one-dimensional) character $e^{\mu}$, where $\mu$ is a (not necessarily dominant) minuscule weight, with the graded character $\gch V_{x}^{-}(\lambda)$ of the level-zero Demazure submodule $V_{x}^{-}(\lambda)$ over the quantum affine algebra $U_{\q}(\Fg_{\af})$ as an explicit finite linear combination of the graded characters of level-zero Demazure submodules.
These identities immediately imply the corresponding inverse Chevalley formulas in the torus-equivariant $K$-group of the semi-infinite flag manifold $\bQG$ associated to a connected, simply-connected and simple algebraic group $G$ of type $C$. Also, we derive cancellation-free identities from the identities above of inverse Chevalley type in the case that $\mu$ is a standard basis element $\ve_{k}$ in the weight lattice $P$ of $G$.
\end{abstract}

\section{Introduction}
The purpose of this paper is to prove identities (of inverse Chevalley type) for the graded characters of Demazure submodules (level-zero Demazure submodules) of extremal weight modules with level-zero extremal weight over a quantum affine algebra of type $C$. 

Let $U_{\q}(\Fg_{\af})$ be the quantum affine algebra associated to the (untwisted) affine Lie algebra $\Fg_{\af}$ whose underlying simple finite-dimensional Lie algebra is $\Fg$.
Let us denote by $W_{\af}$ (resp., $W$) the Weyl group, by $\Fh_{\af}$ (resp., $\Fh$) the Cartan subalgebra, and by $P_{\af}$ (resp., $P$) the weight lattice of $\Fg_{\af}$ (resp., $\Fg$),
where $P = \sum_{i \in I} \BZ \varpi_{i}$ and $P_{\af} = P + \BZ \delta + \BZ \Lambda_{0}$.
For $x \in W_{\af}$ and $\lambda \in P^{+}$, with $P^{+} \subset P$ the set of dominant weights for $\Fg$, let $V_{x}^{-}(\lambda)$ denote the Demazure submodule (level-zero Demazure submodule) of the extremal weight module $V(\lambda)$ with extremal weight $\lambda$ over $U_{\q}(\Fg_{\af})$,
where $\lambda \in P$ is regarded as an element of $P_{\af}$ in a canonical way.
In recent years, the graded characters $\gch V_{x}^{-}(\lambda)$ of the level-zero Demazure submodule $V_{x}^{-}(\lambda)$ for $x \in W_{\af}$, $\lambda \in P$ has been studied in several works. Among them, Kato-Naito-Sagaki \cite{KNS} obtained an explicit description of the expansion of the graded character $\gch V_{x}^{-}(\lambda + \mu)$ for $\lambda, \mu \in P^{+}$ as an infinite linear combination with coefficients in $\BZ[q^{-1}][P]$ of graded characters $\gch V_{y}^{-}(\nu)$ for $y \in W_{\af}$ and $\nu \in P$. Also, Naito-Orr-Sagaki \cite{NOS} obtained a similar description of the graded character $\gch V_{x}^{-}(\lambda - \mu)$ for $\lambda, \mu \in P^{+}$ such that $\lambda - \mu \in P^{+}$; note that in this case, the expansion is, in fact, a finite linear combination with coefficients in $\BZ[q,q^{-1}][P]$.
Recently, Kouno-Lenart-Naito \cite{KLN} (cf.~\cite{LNS}) obtained an explicit description of the expansion, as an infinite linear combination with coefficients in $\BZ[q, q^{-1}][P]$, of the graded character $\gch V_{x}^{-}(\lambda + \mu)$ for $\lambda \in P^{+}$ and an arbitrary $\mu \in P$ such that $\lambda + \mu \in P^{+}$.
This identity is of the following form: 
\begin{equation}\label{eq:Chevalley_intro}
\gch V_{x}^{-}(\lambda + \mu) = \sum_{y \in W_{\af}, \ \nu \in P} c_{x, \mu}^{y, \nu} e^{\nu} \gch V_{y}^{-}(\lambda), 
\end{equation}
where $c_{x, \mu}^{y, \nu} \in \BZ[q, q^{-1}]$ for $y \in W_{\af}$ and $\nu \in P$, and $e^{\nu}$ for $\nu \in P$ denotes the character of $H$ with weight $\nu$. 
Here we should mention that the coefficients $c_{x, \mu}^{y, \nu}$ are independent of the weight $\lambda \in P$; also, for each $y \in W_{\af}$, the sum $\sum_{\nu \in P} c_{x, \mu}^{y, \nu} e^{\nu}$ is an element of $\BZ[q, q^{-1}][P]$.
This explicit identity is called \emph{an identity of Chevalley type}. 

Our main interest lies in an explicit description of the expansion of the product $e^{\nu} \gch V_{x}^{-}(\lambda)$ as a finite linear combination of the graded characters $\gch V_{y}^{-}(\lambda+ \mu)$ for $y \in W_{\af}$ and $\mu \in P$; that is, an explicit description of the coefficients $d_{x, \nu}^{y, \mu}$ in the identity of the following form: 
\begin{equation}\label{eq:inv_Chevalley_intro}
e^{\nu} \gch V_{x}^{-}(\lambda) = \sum_{y \in W_{\af}, \ \mu \in P} d_{x, \nu}^{y, \mu} \gch V_{y}^{-}(\lambda + \mu), 
\end{equation}
where the coefficients $d_{x, \nu}^{y, \mu} \in \BZ[q, q^{-1}]$ are independent of the weight $\lambda \in P$.
In types $A$, $D$, $E_{6}$, $E_{7}$, Kouno-Naito-Orr-Sagaki \cite{KNOS} (for minuscule weights $\nu$) and Lenart-Naito-Orr-Sagaki \cite{LNOS} (for arbitrary weights $\nu$) gave an explicit description of the coefficients $d_{x, \nu}^{y, \mu}$ in the identity above; strictly speaking, the identities obtained in these works are ones in the equivariant $K$-group of the semi-infinite flag manifold $\mathbf{Q}_{G}$ associated to the connected, simply-connected and simple algebraic group $G$ over $\BC$ whose Lie algebra is $\Fg$.
In particular, these identities imply the following \emph{finiteness} result: (i) the right-hand side of the identity \eqref{eq:inv_Chevalley_intro} is a finite sum, and (ii) $d_{x, \nu}^{y, \mu} \in \BZ[q, q^{-1}]$ for all $y \in W_{\af}$ and $\mu \in P$. 
Note that this finiteness result was obtained in simply-laced types by Orr \cite{O}, but the argument therein does not seem to work in non-simply-laced types. 
Since the identity \eqref{eq:inv_Chevalley_intro} can be thought of as an ``inverse expansion'' of the identity \eqref{eq:Chevalley_intro}, we call it \emph{an identity of inverse Chevalley type}. 

In this paper, we study identities of inverse Chevalley type in type $C_{n}$. We give an explicit description of the coefficients $d_{x, \nu}^{y, \mu}$ in the case that $\nu = v \vpi_{1}$ for $v \in W$, where $\vpi_{1}$ is the first fundamental weight. Note that the $W$-orbit of $\vpi_{1}$ is $\{ \pm \ve_{k} \mid k \in \{1, \ldots, n\} \}$, where $\{ \ve_{1}, \ldots, \ve_{n} \}$ is the standard basis of the weight lattice $P \cong \BZ^{n}$; 
for any $v, w \in W$, there exists $m = 1, \ldots, n$ such that $v\vpi_{1} = w\ve_{m}$ or $v\vpi_{1} = -w\ve_{m}$. 

Now we are ready to give the main results of this paper; for the notation used in the following theorems, see Section~\ref{sec:identity_statement}. 
First, we state the ``first half'' of the desired identities of inverse Chevalley type. 
\begin{ithm}[= Corollary~\ref{cor:IC_first-half_complete}] \label{thm:IC_first-half_intro}
For $x = wt_{\xi} \in W_{\af}$ with $w \in W$ and $\xi \in Q^{\vee}$, $m = 1, \ldots, n$, and $\lambda \in P^{+}$ such that $\lambda + \ve_{k} \in P^{+}$ for all $k = 1, \ldots, m$, there holds the following identity: 
\begin{equation*}
\begin{split}
& e^{w\ve_{m}} \gch V_{x}^{-}(\lambda) \\ 
&= q^{\pair{\ve_{m}}{\xi}} \sum_{B \in \CA(w, \Gamma_{m}(m))} (-1)^{|B|} \gch V_{\ed(B)t_{\down(B)+\xi}}^{-}(\lambda + \ve_{m}) \\ 
& \quad + \sum_{j = 1}^{m-1} \sum_{(j_{1}, \ldots, j_{r}) \in \CS_{m, j}} \sum_{A_{1} \in \CA_{w}^{m, j_{1}}} \cdots \sum_{A_{r} \in \CA_{\ed(A_{r-1})}^{j_{r-1}, j_{r}}} (-1)^{|A_{1}|+\cdots+|A_{r}|-r} q^{\pair{\ve_{j}}{\down(A_{1})+\cdots+\down(A_{r})+\xi}} \\ 
& \quad \times \sum_{B \in \CA(\ed(A_{r}), \Gamma_{j}(j))} (-1)^{|B|} \gch V_{\ed(B)t_{\down(B)+\down(A_{1})+\cdots+\down(A_{r})+\xi}}^{-}(\lambda + \ve_{j}). 
\end{split}
\end{equation*}
\end{ithm}
Note that $x$ and $w$ in Theorem~\ref{thm:IC_first-half_intro} agree in the sense that $x = wt_{\xi}$, but $m$ is arbitrary. 
Next, we state the ``second half'' of the desired identities of inverse Chevalley type. 
\begin{ithm}[= Corollary~\ref{cor:IC_second-half_complete}] \label{thm:IC_second-half_intro}
For $x = wt_{\xi} \in W_{\af}$ with $w \in W$ and $\xi \in Q^{\vee}$, $m = 1, \ldots, n$, and $\lambda \in P^{+}$ such that $\lambda + \ve_{k} \in P^{+}$ for all $k = 1, \ldots, n$ and $\lambda - \ve_{k} \in P^{+}$ for $k = m+1, \ldots, n$, there holds the following identity: 
\begin{equation*}
\begin{split}
& e^{-w\ve_{m}} \gch V_{x}^{-}(\lambda) \\ 
&= q^{-\pair{\ve_{m}}{\xi}} \sum_{B \in \CA(w, \Theta_{m})} (-1)^{|B|} \gch V_{\ed(B)t_{\down(B)+\xi}}^{-}(\lambda - \ve_{m}) \\ 
& \quad + \sum_{j = m+1}^{n} \sum_{(j_{1}, \ldots, j_{r}) \in \CS_{\overline{m}, \overline{j}}} \sum_{A_{1} \in \CA_{w}^{\overline{m}, j_{1}}} \cdots \sum_{A_{r} \in \CA_{\ed(A_{r-1})}^{j_{r-1}, j_{r}}} (-1)^{|A_{1}|+\cdots+|A_{r}|-r} q^{-\pair{\ve_{j}}{\down(A_{1})+\cdots+\down(A_{r})+\xi}} \\ 
& \quad \times \sum_{B \in \CA(\ed(A_{r}), \Theta_{j})} (-1)^{|B|} \gch V_{\ed(B)t_{\down(B)+\down(A_{1})+\cdots+\down(A_{r})+\xi}}^{-}(\lambda - \ve_{j}) \\ 
& \quad + \sum_{j = 1}^{n} \sum_{(j_{1}, \ldots, j_{r}) \in \CS_{\overline{m}, j}} \sum_{A_{1} \in \CA_{w}^{\overline{m}, j_{1}}} \cdots \sum_{A_{r} \in \CA_{\ed(A_{r-1})}^{j_{r-1}, j_{r}}} (-1)^{|A_{1}|+\cdots+|A_{r}|-r} q^{\pair{\ve_{j}}{\down(A_{1})+\cdots+\down(A_{r})+\xi}} \\ 
& \quad \times \sum_{B \in \CA(\ed(A_{r}), \Gamma_{j}(j))} (-1)^{|B|} \gch V_{\ed(B)t_{\down(B)+\down(A_{1})+\cdots+\down(A_{r})+\xi}}^{-}(\lambda + \ve_{j}). 
\end{split}
\end{equation*}
\end{ithm}

Note that the proofs of Theorems~\ref{thm:IC_first-half_intro} and \ref{thm:IC_second-half_intro} begin with auxiliary identities (Propositions~\ref{prop:key_first-half} and \ref{prop:key_second-half}) derived directly from a special case of the Chevalley formula given by Proposition~\ref{prop:Chevalley_ek}. 
Also, observe that from the description of these identities, the finiteness result (i), (ii) mentioned above immediately follows,
since every weight $\lambda \in P$ can be written as a $\BZ$-linear combination of $\ve_{1}, \ldots, \ve_{n}$.

Furthermore, we give cancellation-free identities of inverse Chevalley type in the ``first-half'' case, i.e., in the case $\nu = v \varpi_{1} = w \ve_{1}, \ldots, w \ve_{n}$. 
The precise statement is as follows; in the following theorem, $\bp_{m, j}(w)$ denotes a suitable directed path in the quantum Bruhat graph (for the definitions, see Section~\ref{sec:cancellation-free_statement}). 
\begin{ithm}[= Corollary~\ref{cor:IC_cancellation-free_first-half_complete}] 
For $x = wt_{\xi} \in W_{\af}$ with $w \in W$ and $\xi \in Q^{\vee}$, $m = 1, \ldots, n$, and $\lambda \in P^{+}$ such that $\lambda + \ve_{k} \in P^{+}$ for all $k = 1, \ldots, m$, there holds the following cancellation-free identity: 
\begin{equation*}
\begin{split}
& e^{w\ve_{m}} \gch V_{x}^{-}(\lambda) \\ 
&= q^{\pair{\ve_{m}}{\xi}} \sum_{B \in \CA(w, \Gamma_{m}(m))} (-1)^{|B|} \gch V_{\ed(B)t_{\down(B)+\xi}}^{-}(\lambda + \ve_{m}) \\ 
& \quad + \sum_{j = 1}^{m-1} q^{\pair{\ve_{j}}{\wt(\bp_{m, j}(w))+\xi}} \sum_{\CA(\ed(\bp_{m, j}(w)), \Gamma_{j}(j))} (-1)^{|B|} \gch V_{\ed(B)t_{\down(B) + \wt(\bp_{m, j}(w)) + \xi}}^{-}(\lambda + \ve_{j}). 
\end{split}
\end{equation*}
\end{ithm}
As for the ``second-half'' case, i.e., the case $\nu = -v \varpi_{1} = -w \ve_{1}, \ldots, -w \ve_{n}$, 
we provide conjectural cancellation-free identities of inverse Chevalley type in Section~\ref{sec:cancellation-free_conjecture}. 

As an application of our identities of inverse Chevalley type, we can prove a formula for equivariant scalar multiplication (i.e., multiplication with the one-dimensional character $e^{\nu}$, $\nu \in P$, of $H$) in the ($H \times \BC^{\ast}$)-equivariant $K$-group $K_{H \times \BC^{\ast}}(\bQG)$ of the semi-infinite flag manifold $\bQG$ associated to $G$. 
To be more precise, let $\bQGr$ denote the semi-infinite flag manifold associated to $G$, that is, the reduced ind-scheme of infinite type whose set of $\BC$-valued points is $G(\BC\pra{z}) / (H(\BC) \cdot N(\BC \pra{z}))$, where $H \subset G$ is a maximal torus with Lie algebra $\Fh$ and $N$ is the unipotent radical of a Borel subgroup $B \supset H$.
For $\lambda \in P$, there exists a line bundle on $\bQGr$ associated to $\lambda$; we denote by $\CO(\lambda)$ the sheaf corresponding to this line bundle. 
Also, there exist semi-infinite Schubert varieties $\bQG(x)$ for $x \in W_{\af}$, which are subvarieties of $\bQGr$; note that $\bQG = \bQG(e)$, with $e \in W_{\af}$ the identity element.
The equivariant $K$-group $K_{H \times \BC^{\ast}}(\bQG)$ is defined to be the $\BZ[q, q^{-1}][P]$-submodule of (the Laurent series, in $q^{-1}$, extension of) the Iwahori-equivariant $K$-group $K^{\prime}_{\mathbf{I} \rtimes \BC^{\ast}}(\bQG)$, introduced in \cite{KNS}, consisting of all ``convergent'' infinite linear combinations with coefficients in $\BZ[q, q^{-1}][P]$ of the semi-infinite Schubert classes $[\CO_{\bQG(x)}]$, $x \in W_{\af}^{\geq 0} := \{ wt_{\xi} \in W_{\af} \mid w \in W, \xi \in Q^{\vee,+} \}$, where ``convergence'' holds in the sense of \cite[Proposition 5.11]{KNS}; here, $Q^{\vee,+} := \sum_{i \in I} \BZ_{\geq 0} \alpha_{i}^{\vee}$ denotes the positive part of the coroot lattice $Q^{\vee} = \sum_{i \in I} \BZ \alpha_{i}^{\vee}$.

Now, following \cite[Sect. 9]{NOS}, we recall how the graded characters of level-zero Demazure submodules over the quantum affine algebra $U_{\q}(\Fg_{\af})$ are related to the equivariant $K$-group $K_{H \times \BC^{\ast}}(\bQG)$ of the semi-infinite flag manifold $\bQG$. 
Let us define $\BC[q, q^{-1}][P]$-modules $\Fun$, $\Funneg$, and $\Funess$ by 
\begin{align*}
\Fun &:= \{ f: P \rightarrow \BC\pra{q^{-1}}[P] \}, \\
\Funneg &:= \left\{ f \in \Fun \ \middle| \ \parbox{13em}{There exists $\gamma \in P$ such that $f(\mu) = 0$ for all $\mu \in \gamma + P^{+}$} \right\}, \\ 
\Funess &:= \Fun / \Funneg.
\end{align*}
Then there exists an injective $\BZ[q, q^{-1}][P]$-module homomorphism $\Phi: K_{H \times \BC^{\ast}}(\bQG) \rightarrow \Funess$ such that 
for the class $[\mathcal{E}] \in K_{H \times \BC^{\ast}}(\bQG)$ of a certain quasi-coherent sheaf $\mathcal{E}$ on $\bQG$, the element $\Phi([\mathcal{E}]) \in \Funess$ is given as: 
\begin{equation*}
P \rightarrow \BC\pra{q^{-1}}[P], \ \lambda \mapsto \sum_{i = 0}^{\infty} (-1)^{i} \gch H^{i}(\bQG, \mathcal{E} \otimes_{\CO_{\bQG}} \CO(\lambda)); 
\end{equation*}
here, $\gch H^{i}(\bQG, \mathcal{E} \otimes_{\CO_{\bQG}} \CO(\lambda))$ for $i \geq 0$ is the graded character of the $i$-th cohomology group $H^{i}(\bQG, \mathcal{E} \otimes_{\CO_{\bQG}} \CO(\lambda))$, which is regarded as an $(H \times \BC^{\ast})$-module. 
Also, it is proved in \cite{KNS} that we can take $\mathcal{E} = \CO_{\bQG(x)}$ for $x \in W_{\af}^{\ge 0}$, and that
\begin{equation*}
\gch H^{i}(\bQG, \CO_{\bQG(x)} \otimes_{\CO_{\bQG}} \CO(\lambda)) = \begin{cases} \gch V_{x}^{-}(-w_{\circ}\lambda) & \text{if $\lambda \in P^{+}$ and $i = 0$, } \\ 0 & \text{otherwise; } \end{cases} 
\end{equation*}
where $w_{\circ}$ denotes the longest element of $W$. 
By making use of these results, we can translate an identity for graded characters of level-zero Demazure submodules into one in the ($H \times \BC^{\ast}$)-equivariant $K$-group $K_{H \times \BC^{\ast}}(\bQG)$. 
Namely, if we have a finite sum of the form \eqref{eq:inv_Chevalley_intro}, then we obtain the following identity in $K_{H \times \BC^{\ast}}(\bQG)$: 
\begin{equation*}
e^{\nu} \cdot [\CO_{\bQG(x)}] = \sum_{y \in W_{\af}, \ \mu \in P} d_{x, \nu}^{y, \mu} [\CO_{\bQG(y)} \otimes_{\CO_{\bQG}} \CO(-w_{\circ}\mu)]. 
\end{equation*}
In particular, our identities for graded characters of inverse Chevalley type yield explicit identities in the $(H \times \BC^{\ast})$-equivariant $K$-group $K_{H \times \BC^{\ast}}(\bQG)$, which we call \emph{inverse Chevalley forumlas}. 

In addition, by the specialization at $q = 1$ (of the coefficients $d_{x, \nu}^{y, \mu}$), we obtain corresponding inverse Chevalley formulas for equivariant scalar multiplication in the $H$-equivariant $K$-group $K_{H}(\bQG)$ of the semi-infinite flag manifold $\bQG$.
Here we mention that in \cite{Kat}, Kato established a $\BZ[P]$-module isomorphism from $K_{H}(\bQG)$ onto the formal completion $QK_{H}(G/B) \otimes_{\BZ[P][Q^{\vee,+}]} \BZ[P]\bra{Q^{\vee,+}}$ of the 
(small) $H$-equivariant quantum $K$-theory $QK_{H}(G/B) = K_{H}(G/B) \otimes_{\BZ[P]} \BZ[P][Q^{\vee,+}]$ of the finite-dimensional flag manifold $G/B$ which sends each semi-infinite Schubert class to the corresponding (opposite) Schubert class, where $\BZ[P][Q^{\vee,+}]$ (resp., $\BZ[P]\bra{Q^{\vee,+}}$) denotes the ring of polynomials (resp., formal power series) with coefficients in $\BZ[P]$ in the Novikov variables $Q_{i} = Q^{\alpha_{i}^{\vee}}$, $i \in I$. Through this $\BZ[P]$-module isomorphism, we obtain inverse Chevalley formulas for equivariant scalar multiplication in $QK_{H}(G/B)$.

This paper is organized as follows. In Section~\ref{sec:preliminaries}, we fix our basic notation, and recall the definitions of the quantum Bruhat graph and quantum alcove model. 
In Section~\ref{sec:Demazure}, we briefly review level-zero Demazure submodules and identities of Chevalley type for their graded characters. 
In Section~\ref{sec:results}, we state identities of inverse Chevalley type. Also, we give the cancellation-free form of the first half of these identities. 
In Section~\ref{sec:proof_IC}, we prove our identities of inverse Chevalley type. 
In Section~\ref{sec:cancellation-free_proof}, we derive the cancellation-free form of our identities of inverse Chevalley type in the first-half case. 

\subsection*{Acknowledgements}
T. K. was partly supported by JPSP Grant-in-Aid for Scientific Research 20J12058 and 22J00874.
S. N. was partly supported by JSPS Grant-in-Aid for Scientific Research (C) 21K03198.
D. O. was partly supported by a Collaboration Grant for Mathematicians from the Simons Foundation 638577. 

\section{Basic setting}\label{sec:preliminaries}
In this section, we fix basic notation, and review the definitions of the quantum Bruhat graph and quantum alcove model. 

\subsection{Lie algebras and root systems}

Let $\Fg$ be a simple Lie algebra over $\BC$ with Cartan subalgebra $\Fh$. 
Let $\Delta \subset \Fh^{\ast} := \mathrm{Hom}_{\BC}(\Fh, \BC)$ be the root system of $\Fg$, $\Delta^{+} \subset \Delta$ the set of positive roots, 
and $\{ \alpha_{i} \}_{i \in I} \subset \Delta^{+}$ the simple roots. 
We denote by $\pair{\cdot}{\cdot}$ the canonical pairing $\Fh^{\ast} \times \Fh \rightarrow \BC$. 
For $\alpha \in \Delta$, we define $\sgn(\alpha) \in \{ 1, -1 \}$ as 
\begin{equation*}
\sgn(\alpha) := \begin{cases} 1 & \text{if $\alpha \in \Delta^{+}$,} \\ -1 & \text{if $\alpha \in -\Delta^{+}$,} \end{cases}
\end{equation*}
and set $|\alpha| := \sgn(\alpha) \alpha \in \Delta^{+}$. 

For $\alpha \in \Delta$, we denote by $\alpha^{\vee} \in \Fh$ the coroot corresponding to $\alpha$, and define the fundamental weights $\vpi_{i}$, $i \in I$, by $\pair{\vpi_{i}}{\alpha_{j}^{\vee}} = \delta_{i,j}$ for $i, j \in I$. 
Let $P := \sum_{i \in I} \BZ \vpi_{i}$ be the weight lattice, 
$Q := \sum_{i \in I} \BZ \alpha_{i}$ the root lattice, and 
$Q^{\vee} := \sum_{i \in I} \BZ \alpha_{i}^{\vee}$ the coroot lattice. 
Elements of $P^{+} := \sum_{i \in I} \BZ_{\ge 0} \vpi_{i} (\subset P)$ are called dominant weights. 
We denote by $\BZ[P] := \sum_{\lambda \in P} \BZ e^{\lambda}$ the group algebra of $P$, where $\{ e^{\lambda} \mid \lambda \in P \}$ is a formal basis with
relations $e^{\lambda}e^{\mu} = e^{\lambda + \mu}$. 
Note that if $G$ is the connected, simply-connected and simple algebraic group over $\BC$ whose Lie algebra is $\Fg$, then the element $e^{\lambda}$ for $\lambda \in P$ also denotes the one-dimensional representation (character) of the maximal torus $H$ of $G$ of weight $\lambda$. 
In particular, $\BZ[P]$ is isomorphic to the representation ring $R(H)$ of the torus $H$. 

For $\alpha \in \Delta$, we define the reflection $s_{\alpha} \in GL(\Fh^{\ast})$ by $s_{\alpha}(\lambda) := \lambda - \pair{\lambda}{\alpha^{\vee}}\alpha$, $\lambda \in \Fh^{\ast}$. 
In particular, the reflection $s_{i} := s_{\alpha_{i}}$ for $i \in I$ is called a simple reflection. 
The Weyl group $W$ is defined as the subgroup of $GL(\Fh^{\ast})$ generated by $\{ s_{i} \}_{i \in I}$, i.e., 
$W = \langle s_{i} \mid i \in I \rangle \subset GL(\Fh^{\ast})$. 

\subsection{Type \texorpdfstring{$C$}{C} root system}
We review the standard realization of the root system of type $C$. Let $\{\ve_{1}, \ldots, \ve_{n}\}$ be the standard basis of $\BR^{n}$. 
Then, the set 
\begin{equation*}
\Delta = \{ \pm (\ve_{i} - \ve_{j}) \mid 1 \le i < j \le n \} \sqcup \{ \pm (\ve_{i} + \ve_{j}) \mid 1 \le i < j \le n \} \sqcup \{ \pm 2\ve_{k} \mid 1 \le k \le n \}
\end{equation*}
forms the root system of type $C_{n}$, and the set 
\begin{equation*}
\Delta^{+} = \{ \ve_{i} - \ve_{j} \mid 1 \le i < j \le n \} \sqcup \{ \ve_{i} + \ve_{j} \mid 1 \le i < j \le n \} \sqcup \{ 2\ve_{k} \mid 1 \le k \le n \}
\end{equation*}
is the set of all positive roots. In particular, $\alpha_{i}$, $i = 1, \ldots, n$, defined by 
\begin{equation*}
\alpha_{i} := \ve_{i} - \ve_{i+1} \ (1 \le i \le n-1), \quad \alpha_{n} := 2\ve_{n}
\end{equation*}
are the simple roots. 

For $1 \le i < j \le n$, we set 
\begin{equation*}
(i, j) := \ve_{i} - \ve_{j}, \quad (i, \overline{j}) := \ve_{i} + \ve_{j}, \quad (i, \overline{i}) := 2\ve_{i}. 
\end{equation*}
The Weyl group $W$ of type $C_{n}$ is realized as a subgroup of the permutation group of the set $[\overline{n}] := \{1, 2, \ldots, n, \overline{n}, \overline{n-1}, \ldots, \overline{1}\}$ 
by identifying simple reflections $s_{1}, \ldots, s_{n-1}, s_{n}$ with transpositions $(1 \ 2), \ldots, (n-1 \ n), (n \ \overline{n})$, respectively. 

\subsection{The quantum Bruhat graph}
The quantum Bruhat graph is a labeled directed graph on the Weyl group $W$, introduced by Brenti-Fomin-Postnikov \cite{BFP}. 

\begin{defn} [{\cite[Definition~6.1]{BFP}}]
The \emph{quantum Bruhat graph} $\QBG(W)$ is the $\Delta^{+}$-labeled directed graph whose vertex set is $W$, and whose edges are given as follows. 
For $x, y \in W$ and $\alpha \in \Delta^{+}$, we have a directed edge $x \xrightarrow{\alpha} y$ if $y = xs_{\alpha}$, and either of the following holds: 
(B) $\ell(y) = \ell(x) + 1$, or (Q) $\ell(y) = \ell(x) - 2 \pair{\rho}{\alpha^{\vee}} + 1$, where $\rho := (1/2)\sum_{\alpha \in \Delta^{+}} \alpha$. 
If the condition (B) (resp., (Q)) holds, then the corresponding edge $x \xrightarrow{\alpha} y$ is called a \emph{Bruhat edge} (resp., \emph{quantum edge}). 
\end{defn}

For a directed path $\bp: w_{0} \xrightarrow{\gamma_{1}} w_{1} \xrightarrow{\gamma_{2}} \cdots \xrightarrow{\gamma_{r}} w_{r}$ in $\QBG(W)$, we define $\wt(\bp) \in Q^{\vee}$ by 
\begin{equation*}
\wt(\bp) := \sum_{\substack{1 \le k \le r \\ \text{$w_{k-1} \xrightarrow{\gamma_{k}} w_{k}$ is a quantum edge}}} \gamma_{k}^{\vee}. 
\end{equation*}

\subsection{The quantum alcove model}
We briefly review the theory of quantum alcove model, first introduced by Lenart-Lubovsky \cite{LL}, and then generalized by \cite{LNS}. 

We set $\Fh_{\BR}^{\ast} := P \otimes_{\BZ} \BR$. For $\alpha \in \Delta$ and $k \in \BZ$, we define a hyperplane $H_{\alpha, k}$ in $\Fh_{\BR}^{\ast}$ by 
\begin{equation*}
H_{\alpha, k} := \{ \xi \in \Fh_{\BR}^{\ast} \mid \pair{\xi}{\alpha^{\vee}} = k \}; 
\end{equation*}
we denote by $s_{\alpha, k}$ the reflection for the hyperplane $H_{\alpha, k}$. Connected components of the space $\Fh_{\BR}^{\ast} \setminus \bigcup_{\alpha \in \Delta, \ k \in \BZ} H_{\alpha, k}$ are called \emph{alcoves}. 
Two alcoves $A$, $B$ are called \emph{adjacent} if the closures of $A$ and $B$ have an intersection, called a \emph{common wall}. 

\begin{defn} [{\cite[Definition~5.2]{LP1}}]
A sequence $(A_{0}, A_{1}, \ldots, A_{r})$ of alcoves $A_{0}, \ldots, A_{r}$ is called an \emph{alcove path} if $A_{i-1}$ and $A_{i}$ are adjacent for each $i = 1, \ldots, r$. An alcove path $\Gamma = (A_{0}, \ldots, A_{r})$ is called \emph{reduced} if $\Gamma$ has a minimal length $r$ among all alcove paths from $A_{0}$ to $A_{r}$. 
\end{defn}

For adjacent alcoves $A$, $B$, and a root $\alpha \in \Delta$, we write $A \xrightarrow{\alpha} B$ if the common wall of $A$ and $B$ is contained in the hyperplane $H_{\alpha, k}$ for some $k \in \BZ$, and $\alpha$ points in a direction from $A$ to $B$ (as a direction vector). 
We take a special alcove $A_{\circ}$, called the \emph{fundamental alcove}, defined by 
\begin{equation*}
A_{\circ} := \{ \xi \in \Fh_{\BR}^{\ast} \mid \text{$0 < \pair{\xi}{\alpha^{\vee}} < 1$ for all $\alpha \in \Delta^{+}$} \}. 
\end{equation*}
For $\lambda \in P$, we define $A_{\lambda}$ by 
\begin{equation*}
A_{\lambda} := A_{\circ} + \lambda = \{ \xi + \lambda \mid \xi \in A_{\circ} \}. 
\end{equation*}

\begin{defn} [{\cite[Definition~5.4]{LP1}}]
Let $\lambda \in P$. A sequence $\Gamma = (\gamma_{1}, \ldots, \gamma_{r})$ of roots $\gamma_{1}, \ldots, \gamma_{r} \in \Delta$ is called a \emph{$\lambda$-chain} if there exists an alcove path $(A_{\circ} = A_{0}, A_{1}, \ldots, A_{r} = A_{-\lambda})$ such that 
\begin{equation*}
A_{\circ} = A_{0} \xrightarrow{-\gamma_{1}} A_{1} \xrightarrow{-\gamma_{2}} \cdots \xrightarrow{-\gamma_{r}} A_{r} = A_{-\lambda}. 
\end{equation*}
We say that $\Gamma$ is \emph{reduced} if the corresponding alcove path $(A_{0}, \ldots, A_{r})$ is reduced.
\end{defn}

Let $\Gamma$ be a sequence of roots, i.e., $\Gamma = (\gamma_{1}, \ldots, \gamma_{r})$, with $\gamma_{k} \in \Delta$, $k = 1, \ldots, r$. 
\begin{defn} [{\cite[Definition~17]{LNS}}]
Let $w \in W$. A subset $A = \{i_{1} < \cdots < i_{s}\} \subset \{1, \ldots, r\}$ is called \emph{$w$-admissible} if
\begin{equation*}
w = w_{0} \xrightarrow{|\gamma_{i_{1}}|} w_{1} \xrightarrow{|\gamma_{i_{2}}|} \cdots \xrightarrow{|\gamma_{i_{s}}|} w_{s}
\end{equation*}
is a directed path in $\QBG(W)$. 
In this case, we define $\ed(A)$ by $\ed(A) := w_{s}$. 
Also, we set 
\begin{equation*}
A^{-} := \{ k \in A \mid \text{the edge $w_{k-1} \xrightarrow{|\gamma_{k}|} w_{k}$ is a quantum edge}\}, 
\end{equation*}
and then define $\down(A)$ by 
\begin{equation*}
\down(A) := \sum_{k \in A^{-}} |\gamma_{k}|^{\vee}. 
\end{equation*}
Also, we set 
\begin{equation*}
n(A) := \# \{ k \in A \mid \gamma_{k} \in -\Delta^{+} \}. 
\end{equation*}
We denote by $\CA(w, \Gamma)$ the set of all $w$-admissible subsets. 
\end{defn}

Let $\Gamma_{1}, \ldots, \Gamma_{r}$ be sequences of roots, and $w \in W$. 
For a tuple $(A_{1}, A_{2} \ldots, A_{r})$ of admissible subsets $A_{1} \in \CA(w, \Gamma_{1}), A_{2} \in \CA(\ed(A_{1}), \Gamma_{2}), \ldots, A_{r} \in \CA(\ed(A_{r-1}), \Gamma_{r})$, 
we set 
\begin{equation*}
\down(A_{1}, A_{2}, \ldots, A_{r}) := \down(A_{1}) + \down(A_{2}) + \cdots + \down(A_{r}). 
\end{equation*}

If $\Gamma$ is a $\lambda$-chain for some $\lambda \in P$, then we can consider additional statistics denoted by $\wt$, $\height$. 
For a $\lambda$-chain $\Gamma = (\gamma_{1}, \ldots, \gamma_{r})$ with $\lambda \in P$,
let $(A_{\circ} = A_{0}, \ldots, A_{r} = A_{-\lambda})$ be the alcove path corresponding to $\Gamma$, 
and take integers $l_{k} \in \BZ$, $k = 1, \ldots, r$, such that the common wall of adjacent alcoves $A_{k-1}$ and $A_{k}$ is contained in the hyperplane $H_{\gamma_{k}, -l_{k}}$. Then, we define $\wt(A)$ and $\height(A)$ for $A = \{i_{1} < \cdots < i_{s}\}$ by 
\begin{equation*}
\wt(A) := -ws_{\gamma_{i_{1}}, -l_{i_{1}}} \cdots s_{\gamma_{i_{s}}, -l_{i_{s}}}(-\lambda), \quad \height(A) := \sum_{k \in A^{-}} \sgn(\gamma_{k}) (\pair{\lambda}{\gamma_{k}^{\vee}} - l_{k}). 
\end{equation*}

\subsection{Specific chains of roots}

In this subsection, we deal with the root system of type $C_{n}$. 
We choose specific $(-\vpi_{k-1}+\vpi_{k})$-chain and $(\vpi_{k-1}-\vpi_{k})$-chain, which will play a crucial role in this paper; 
we understand that $\vpi_{0} = 0$ in this paper. 
Note that $-\vpi_{k-1}+\vpi_{k} = \ve_{k}$. 
We set 
\begin{align*}
\begin{split}
\Gamma_{k}(k) &:= (-(1, \overline{k}), \ldots, -(k-1, \overline{k}), \\ 
& \quad \quad -(k, \overline{k+1}), \ldots, -(k, \overline{n}), \\ 
& \quad \quad -(k, \overline{k}), \\ 
& \quad \quad -(k, n), \ldots, (k, k+1)), 
\end{split} \\ 
\begin{split}
\Gamma_{k}^{\ast}(k) &:= ((k, k+1), \ldots, (k,n), \\ 
& \quad \quad (k, \overline{k}), \\ 
& \quad \quad (k, \overline{n}), \ldots, (k, \overline{k+1}), \\ 
& \quad \quad (k-1, \overline{k}), \ldots, (1, \overline{k})), 
\end{split} \\
\Theta_{k} &:= (-(1, k), \ldots, -(k-1, k)), \\ 
\Theta_{k}^{\ast} &:= ((k-1, k), \ldots, (1, k)). 
\end{align*}

For sequences $\Gamma = (\gamma_{1}, \ldots, \gamma_{r}), \Xi = (\xi_{1}, \ldots, \xi_{s})$ of roots, we denote by $\Gamma \ast \Xi$ the concatenation of $\Gamma$ and $\Xi$, i.e., $\Gamma \ast \Xi := (\gamma_{1}, \ldots, \gamma_{r}, \xi_{1}, \ldots, \xi_{s})$. 

\begin{lem} \label{lem:ek-chain}
The concatenation $\Gamma_{k-1, k} := \Gamma_{k}^{\ast}(k) \ast \Theta_{k}$ is a reduced $(-\vpi_{k-1}+\vpi_{k})$-chain. 
\end{lem}
\begin{proof}
We set $x := s_{k}s_{k+1} \cdots s_{n} s_{n-1} \cdots s_{1}$, $y := s_{1} \cdots s_{k-1}$, and $\mu := \vpi_{1}$. 
Then, $x$ is a minimal-length representative for the coset $xW_{\mu}$, where $W_{\mu} := \{ w \in W \mid w\mu = \mu\}$, $yx$ is the minimal-length representative for the coset $\{ w \in W \mid w\mu = w_{\circ}\mu\}$, and $x\mu = -(-\vpi_{k-1} + \vpi_{k})$. 
Now, following \cite[Lemma~4.1]{LNOS}, we define $\Gamma$ as follows. 
Let us write $x = s_{j_{a}} \cdots s_{j_{1}}$, $y = s_{i_{1}} \cdots s_{i_{b}}$, and set 
\begin{align*}
\beta_{c} &:= s_{j_{a}} \cdots s_{j_{c+1}} \alpha_{j_{c}}, \quad 1 \le c \le a, \\ 
\zeta_{d} &:= s_{i_{b}} \cdots s_{i_{d+1}} \alpha_{i_{d}}, \quad 1 \le d \le b. 
\end{align*}
Then we define $\Gamma$ as $\Gamma := (\beta_{1}, \ldots, \beta_{a}, -\gamma_{1}, \ldots, -\gamma_{b})$; 
note that the convention for the sign of roots in alcove paths in this paper is different from that of \cite{LNOS}. 
By direct calculation, we see that this $\Gamma$ is identical to $\Gamma_{k-1, k}$. 
Since $\mu$ is a minuscule fundamental weight, the argument in the proof of \cite[Lemma~4.1]{LNOS} still works in our setting of the type $C$ root system, 
and hence we obtain the following reduced alcove path $\Pi$ from $A_{\circ}$ to $A_{\circ} + x\mu = A_{\circ} - (-\vpi_{k-1} + \vpi_{k}) = A_{-(-\vpi_{k-1} + \vpi_{k})}$: 
\begin{equation*}
\begin{split}
\Pi:{} & A_{\circ} = A_{0} \xrightarrow{-\beta_{1}} A_{1} \xrightarrow{-\beta_{2}} \cdots \xrightarrow{-\beta_{a}} A_{a} = B_{0} \\ 
& \xrightarrow{\zeta_{1}} B_{1} \xrightarrow{\zeta_{2}} \cdots \xrightarrow{\zeta_{b}} B_{b} = A_{-(-\vpi_{k-1} + \vpi_{k})}. 
\end{split}
\end{equation*}
Thus we have shown that $\Gamma_{k-1, k}$ is a reduced $(-\vpi_{k-1}+\vpi_{k})$-chain corresponding to $\Pi$. 
This proves the lemma. 
\end{proof}

\begin{rem} \label{rem:common_wall}
The proof of \cite[Lemma~4.1]{LNOS} also shows that for $t = 1, \ldots, a$, the common wall of the adjacent alcoves $A_{t-1}$ and $A_{t}$ in the above path $\Pi$ 
is contained in the hyperplane $H_{\beta_{t}, 0}$, 
while for $t = 1, \ldots, b$, the common wall of the adjacent alcoves $B_{t-1}$ and $B_{t}$ is contained in the hyperplane $H_{\zeta_{t}, 1}$. 
\end{rem}

By reversing the order of roots in $\Gamma$ and negating all roots, we obtain a specific $(\vpi_{k-1} - \vpi_{k})$-chain. 
\begin{cor} \label{cor:-ek-chain}
The concatenation $\Gamma_{k-1, k}^{\ast} := \Theta_{k}^{\ast} \ast \Gamma_{k}(k)$ is a reduced $(\vpi_{k-1} - \vpi_{k})$-chain. 
\end{cor}

\section{Level-zero Demazure submodules over quantum affine algebras}\label{sec:Demazure}
We recall the definition of level-zero Demazure submodules over quantum affine algebras and their graded characters. 

\subsection{Notation for affine Lie algebras and quantum affine algebras}
Let $\Fg_{\af} := (\Fg \otimes_{\BC} \BC [t, t^{-1}]) \oplus \BC c \oplus \BC d$ be the (untwisted) affine Lie algebra associated to $\Fg$, 
where $c$ is the canonical central element and $d$ is the scaling element. We denote by $\Fh_{\af}$ its Cartan subalgebra. 
Let $\pair{\cdot}{\cdot}$ be the canonical pairing $\Fh_{\af}^{\ast} \times \Fh_{\af} \rightarrow \BC$, 
where $\Fh_{\af}^{\ast} = \mathrm{Hom}_{\BC}(\Fh_{\af}, \BC)$. 
We set $I_{\af} := I \sqcup \{0\}$. Then, the simple roots $\alpha_{i}$, $i \in I \subsetneq I_{\af}$, of $\Fg$ can be regarded as simple roots of $\Fg_{\af}$. Let $s_{i}$, $i \in I_{\af}$, be the simple reflections corresponding to $\alpha_{i}$. 
Let $W_{\af} := \langle s_i \mid i \in I_{\af} \rangle$ denote the (affine) Weyl group of $\Fg_{\af}$. 
We know that $W_{\af} = \{ wt_{\xi} \mid w \in W, \ \xi \in Q^{\vee} \} \simeq W \ltimes Q^{\vee}$, where $t_{\xi}$, $\xi \in Q^{\vee}$, is the translation element (\cite{Kac}). 

Let $U_{\q}(\Fg_{\af})$ be the quantum affine algebra associated to $\Fg_{\af}$, 
and denote by $E_{i}, F_{i}$, $i \in I_{\af} = I \sqcup \{0\}$, the Chevalley generators of $U_{\q}(\Fg_{\af})$. 
Then, we define $U_{\q}^{-}(\Fg_{\af})$ as the subalgebra of $U_{\q}(\Fg_{\af})$ generated by $\{ F_{i} \mid i \in I_{\af} \}$, i.e., $U_{\q}^{-}(\Fg_{\af}) = \langle F_{i} \mid i \in I_{\af} \rangle$. 

\subsection{Extremal weight submodules and level-zero Demazure submodules}

\begin{defn} [{\cite[Definition~8.1.1]{Kas}}] \label{def:extremal}
Let $M$ be an integrable $U_{\q}(\Fg_{\af})$-module, and $\lambda \in P_{\af}$. An element $v \in M$ is called an \emph{extremal weight vector of weight $\lambda$} if $v$ is a weight vector of $\lambda$, and there exists a family $\{v_{x} \mid x \in W_{\af}\} \subset M$ of vectors such that 
\begin{enu}
\item $v_{e} = v$, 
\item for $i \in I_{\af}$ and $x \in W_{\af}$, if $\pair{x\lambda}{\alpha_{i}^{\vee}} \ge 0$, then $E_{i}v_{x} = 0$ and $F_{i}^{(\pair{x\lambda}{\alpha_{i}^{\vee}})}v_{x} = v_{s_{i}x}$, and 
\item for $i \in I_{\af}$ and $x \in W_{\af}$, if $\pair{x\lambda}{\alpha_{i}^{\vee}} \le 0$, then $F_{i}v_{x} = 0$ and $E_{i}^{(-\pair{x\lambda}{\alpha_{i}^{\vee}})}v_{x} = v_{s_{i}x}$, 
\end{enu}
where $F_{i}^{(k)}$ and $E_{i}^{(k)}$, $k \ge 0$, denote the divided powers. 
\end{defn}

For $\lambda \in P_{\af}$, the \emph{extremal weight module} of weight $\lambda$, denoted by $V(\lambda)$, is the integrable weight module over $U_{\q}(\Fg_{\af})$ whose generator is a single element $v_{\lambda}$, and whose defining relation is that ``$v_{\lambda}$ is an extremal weight vector of weight $\lambda$''; for the precise definition of extremal weight modules, see \cite[Proposition~8.2.2]{Kas}. 

Let $\lambda \in P^{+} \subset P_{\af}$, and $x \in W_{\af}$. 
By the definition of extremal weight vectors, there exists a family $\{v_{x} \mid x \in W_{\af}\} \subset V(\lambda)$ of vectors satisfying the conditions in Definition~\ref{def:extremal}, with $v = v_{\lambda}$. 
The \emph{level-zero Demazure submodule} $V_{x}^{-}(\lambda)$ is a $U_{\q}^{-}(\Fg_{\af})$-submodule of $V(\lambda)$ generated by $v_{x}$, i.e., $V_{x}^{-}(\lambda) = U_{\q}^{-}(\Fg_{\af}) v_{x}$. 
For $\nu \in P_{\af}$, we denote by $V_{x}^{-}(\lambda)_{\nu}$ the weight space of $V_{x}^{-}(\lambda)$ of weight $\nu \in P_{\af}$. 
Then we have the following weight space decomposition with respect to $\Fh_{\af}$: 
\begin{equation*}
V_{x}^{-}(\lambda) = \bigoplus_{\gamma \in Q, \ k \in \BZ} V_{x}^{-}(\lambda)_{\lambda + \gamma + k\delta}, 
\end{equation*}
where each weight space $V_{x}^{-}(\lambda)_{\lambda + \gamma + k\delta}$, $\gamma \in Q$, $k \in \BZ$, is a finite-dimensional $\BC(\q)$-vector space; 
here, $\delta$ denotes the (primitive) null root of $\Fg_{\af}$.
Now we define the \emph{graded character} of $V_{x}^{-}(\lambda)$ by 
\begin{equation*}
\gch V_{x}^{-}(\lambda) := \sum_{\gamma \in Q, \ k \in \BZ} \dim(V_{x}^{-}(\lambda)_{\lambda + \gamma + k\delta}) q^{k} e^{\lambda + \gamma} \in \BZ[P]\pra{q^{-1}}, 
\end{equation*}
where $q$ is an indeterminate (not to be confused with $\q$). 

The following identity is useful to compute graded characters of level-zero Demazure submodules.
\begin{prop} [{\cite[Proposition~D.1]{KNS}}] \label{prop:gch_translation}
Let $x \in W_{\af}$ and $\lambda \in P^{+}$. For $\xi \in Q^{\vee}$, we have 
\begin{equation*}
\gch V_{xt_{\xi}}^{-}(\lambda) = q^{-\pair{\lambda}{\xi}} \gch V_{x}^{-}(\lambda). 
\end{equation*}
\end{prop}

\subsection{Identities of Chevalley type}
Let $\lambda \in P^{+}$ and $x \in W_{\af}$.
We consider the graded character $\gch V_{x}^{-}(\lambda + \mu)$, where $\mu \in P$ is such that $\lambda + \mu \in P^{+}$. For this, we need to introduce more notation. 
A \emph{partition} is a weakly decreasing sequence $\chi = (\chi_{1} \ge \cdots \ge \chi_{l})$ of positive integers $\chi_{1}, \ldots, \chi_{l} \in \BZ_{> 0}$; we call $l$ the \emph{length} of $\chi$.
Also, we set $|\chi| := \chi_{1} + \cdots + \chi_{l}$, the \emph{size} of $\chi$. 
If $\chi = \emptyset$, the empty partition, then we set $\ell(\chi) := 0$, and $|\chi| := 0$. 
Let $\mu \in P$, and write $\mu = \sum_{i \in I} m_{i} \vpi_{i}$. We define a set $\bPar(\mu)$ as follows: 
\begin{equation*}
\bPar(\mu) := \{ \bchi = (\chi^{(i)})_{i \in I} \mid \text{$\chi^{(i)}$, $i \in I$, are partitions such that $\ell(\chi^{(i)}) \le \max \{m_{i}, 0\}$} \}. 
\end{equation*}
For $\bchi = (\chi^{(i)})_{i \in I} \in \bPar(\mu)$, we write $\chi^{(i)} = (\chi_{1}^{(i)} \ge \cdots \ge \chi_{l_{i}}^{(i)})$, where $\chi_{1}^{(i)}, \ldots, \chi_{l_{i}}^{(i)} \in \BZ_{> 0}$, with $l_{i} = \ell(\chi^{(i)})$. We set 
\begin{equation*}
|\bchi| := \sum_{i \in I} |\chi^{(i)}|, \quad \iota(\bchi) := \sum_{i \in I} \chi_{1}^{(i)} \alpha_{i}^{\vee} \in Q^{\vee, +}, 
\end{equation*}
where, if $\chi^{(i)} = \emptyset$, then we set $\chi_{1}^{(i)} := 0$. 

Lenart-Naito-Sagaki \cite{LNS} and Kouno-Lenart-Naito \cite{KLN} proved the following identity, called \emph{the identity of Chevalley type}. 
\begin{thm} [{\cite[Theorem~33]{LNS} and \cite[Theorem~5.16]{KLN}}] \label{thm:Chevalley}
Let $\lambda \in P^{+}$, $\mu \in P$, and $x = wt_{\af} \in W_{\af}$ with $w \in W$ and $\xi \in Q^{\vee}$. Assume that $\lambda + \mu \in P^{+}$. 
Take a reduced $\mu$-chain $\Gamma$. Then, there holds the following identity: 
\begin{equation} \label{eq:Chevalley_LNS}
\begin{split}
& \gch V_{x}^{-}(\lambda + \mu) \\
&= \sum_{A \in \CA(w, \Gamma)} \sum_{\bchi \in \bPar(\mu)} (-1)^{n(A)} q^{-\height(A)-\pair{\lambda}{\xi}-|\bchi|} e^{\wt(A)} \gch V_{\ed(A)t_{\xi+\down(A)+\iota(\bchi)}}^{-}(\lambda). 
\end{split}
\end{equation}
\end{thm}

\begin{rem}
Strictly speaking, Lenart-Naito-Sagaki proved an identity, called the Chevalley formula, in the equivariant $K$-group of semi-infinite flag manifolds, which is equivalent to \eqref{eq:Chevalley_LNS}. 
\end{rem}

Now, we consider the root system of type $C_{n}$, and apply the identity of Chevalley type above to the case that $\mu = -\vpi_{k-1} + \vpi_{k} = \ve_{k}$, $k = 1, \ldots, n$, to obtain the following.

\begin{prop} \label{prop:Chevalley_ek}
Let $2 \le k \le n$ and $\mu := \ve_{k} = -\vpi_{k-1} + \vpi_{k}$. Take an arbitrary reduced $\mu$-chain $\Gamma$. Let $w \in W$. For $\lambda \in P^{+}$ such that $\lambda + \mu \in P^{+}$, we have 
\begin{equation*}
\gch V_{w}^{-}(\lambda + \mu) = \frac{1}{1 - q^{-\pair{\lambda + \vpi_{k}}{\alpha_{k}^{\vee}}}} \sum_{A \in \CA(w, \Gamma)} (-1)^{n(A)} q^{-\height(A)}e^{\wt(A)} \gch V_{\ed(A)t_{\down(A)}}^{-}(\lambda). 
\end{equation*}
\end{prop}
\begin{proof}
Since $\mu = -\vpi_{k-1} + \vpi_{k}$, we have 
\begin{equation*}
\bPar(\mu) = \{ \bi_{k} := (\emptyset, \ldots, \emptyset, \underset{k}{(i)}, \emptyset, \ldots, \emptyset) \mid i \ge 0 \}, 
\end{equation*}
where $\emptyset$ denotes the empty partition (of length $0$), which is also regard as $(0)$. 
For $i \ge 0$, we have $|\bi_{k}| = i$, and $\iota(\bi_{k}) = i \alpha_{k}^{\vee}$. 
Therefore, by Theorem~\ref{thm:Chevalley}, we compute:
\begin{align*}
\gch V_{w}^{-}(\lambda + \mu) &= \sum_{A \in \CA(w, \Gamma)} \sum_{\bchi \in \bPar(\mu)} (-1)^{n(A)} q^{-\height(A) - |\bchi|} e^{\wt(A)} \gch V_{\ed(A)t_{\down(A)+\iota(\bchi)}}^{-}(\lambda) \\ 
&= \sum_{A \in \CA(w, \Gamma)} \sum_{i = 0}^{\infty} (-1)^{n(A)} q^{-\height(A)-i} e^{\wt(A)} \gch V_{\ed(A)t_{\down(A)+i\alpha_{k}^{\vee}}}^{-}(\lambda) \\ 
&= \sum_{A \in \CA(w, \Gamma)} \sum_{i = 0}^{\infty} (-1)^{n(A)} q^{-\height(A)-i} e^{\wt(A)} q^{-\pair{\lambda}{i\alpha_{k}^{\vee}}} \gch V_{\ed(A)t_{\down(A)}}^{-}(\lambda) \\ 
&= \sum_{i = 0}^{\infty} q^{-i-\pair{\lambda}{i\alpha_{k}^{\vee}}} \sum_{A \in \CA(w, \Gamma)} (-1)^{n(A)} q^{-\height(A)} e^{\wt(A)} \gch V_{\ed(A)t_{\down(A)}}^{-}(\lambda) \\ 
&= \sum_{i = 0}^{\infty} q^{-i\pair{\lambda+\vpi_{k}}{\alpha_{k}^{\vee}}} \sum_{A \in \CA(w, \Gamma)} (-1)^{n(A)} q^{-\height(A)} e^{\wt(A)} \gch V_{\ed(A)t_{\down(A)}}^{-}(\lambda) \\ 
&= \frac{1}{1-q^{-\pair{\lambda+\vpi_{k}}{\alpha_{k}^{\vee}}}} \sum_{A \in \CA(w, \Gamma)} (-1)^{n(A)} q^{-\height(A)} e^{\wt(A)} \gch V_{\ed(A)t_{\down(A)}}^{-}(\lambda), 
\end{align*}
as desired; for the third equality, we used Proposition~\ref{prop:gch_translation}. 
This proves the proposition. 
\end{proof}

\section{Main Results} \label{sec:results}
In this section, we give precise statements of identities of inverse Chevalley type in type $C_{n}$. 
First, we give identities in which some terms may cancel. 
Next, we describe the cancellations in the ``first half'' of these identities to obtain cancellation-free ones.
Also, we make a conjecture for the cancellations in the ``second half'' of these identities. 
In the rest of this paper, we assume that $\mathfrak{g}$ is of type $C_{n}$. 

\subsection{Identities of inverse Chevalley type} \label{sec:identity_statement}
To give precise statements of our main results, we prepare additional notation. 
Let us define a total order $<$ on the set $[\overline{n}]$ by $1 < 2 < \cdots < n < \overline{n} < \overline{n-1} < \cdots < \overline{1}$. 
For $j, m \in [\overline{n}]$ with $j < m$, we define $\CS_{m, j}$ to be the set of all strictly decreasing sequences of integers starting at $m$ and ending at $j$, that is, 
\begin{equation*}
\CS_{m, j} := \{ (j_{1}, \ldots, j_{r}) \mid r \ge 1, \ j_{1}, \ldots, j_{r} \in [\overline{n}], \ m > j_{1} > \cdots > j_{r} = j \}. 
\end{equation*}
For $w \in W$ and $1 \le l < k \le n$, we set 
\begin{equation*}
\CA_{w}^{k, l} := \{ A \in \CA(w, \Theta_{k}) \setminus \{ \emptyset \} \mid \ed(A)^{-1}w\ve_{k} = \ve_{l} \}. 
\end{equation*}
Also, for $w \in W$, $k \in \{1, \ldots, n\}$, and $l < \overline{k}$, we set 
\begin{equation*}
\CA_{w}^{\overline{k}, l} := \{ A \in \CA(w, \Gamma_{k}(k)) \setminus \{ \emptyset \} \mid \ed(A)^{-1} w (-\ve_{k}) = \ve_{l} \}, 
\end{equation*}
where for simplicity, we write $\ve_{\overline{m}} = -\ve_{m}$ for $m \in \{1, \ldots, n\}$. 

The following is the ``first half'' of identities of inverse Chevalley type (in which cancellations occur). 

\begin{thm} \label{thm:IC_first-half}
For $w \in W$, $\lambda \in P^{+}$, $m = 1, \ldots, n$, and $\lambda \in P^{+}$ such that $\lambda + \ve_{k} \in P^{+}$ for all $k = 1, \ldots, m$, there holds the following identity: 
\begin{equation} \label{eq:IC_first-half}
\begin{split}
& e^{w\ve_{m}} \gch V_{w}^{-}(\lambda) \\ 
&= \sum_{B \in \CA(w, \Gamma_{m}(m))} (-1)^{|B|} \gch V_{\ed(B)t_{\down(B)}}^{-}(\lambda + \ve_{m}) \\ 
& \quad + \sum_{j = 1}^{m-1} \sum_{(j_{1}, \ldots, j_{r}) \in \CS_{m, j}} \sum_{A_{1} \in \CA_{w}^{m, j_{1}}} \cdots \sum_{A_{r} \in \CA_{\ed(A_{r-1})}^{j_{r-1}, j_{r}}} (-1)^{|A_{1}|+\cdots+|A_{r}|-r} q^{\pair{\ve_{j}}{\down(A_{1}, \ldots, A_{r})}} \\ 
& \quad \times \sum_{B \in \CA(\ed(A_{r}), \Gamma_{j}(j))} (-1)^{|B|} \gch V_{\ed(B)t_{\down(A_{1}, \ldots, A_{r}, B)}}^{-}(\lambda + \ve_{j}). 
\end{split}
\end{equation}
\end{thm}

We give a proof of Theorem~\ref{thm:IC_first-half} in Section~\ref{sec:proof_IC}. 
By Proposition~\ref{prop:gch_translation}, we obtain the following identities for an arbitrary $x \in W_{\af}$ (not only for $x \in W$). 
\begin{cor} \label{cor:IC_first-half_complete}
For $x = wt_{\xi} \in W_{\af}$ with $w \in W$ and $\xi \in Q^{\vee}$, $m = 1, \ldots, n$, and $\lambda \in P^{+}$ such that $\lambda + \ve_{k} \in P^{+}$ for all $k = 1, \ldots, m$, there holds the following identity: 
\begin{equation*}
\begin{split}
& e^{w\ve_{m}} \gch V_{x}^{-}(\lambda) \\ 
&= q^{\pair{\ve_{m}}{\xi}} \sum_{B \in \CA(w, \Gamma_{m}(m))} (-1)^{|B|} \gch V_{\ed(B)t_{\down(B)+\xi}}^{-}(\lambda + \ve_{m}) \\ 
& \quad + \sum_{j = 1}^{m-1} \sum_{(j_{1}, \ldots, j_{r}) \in \CS_{m, j}} \sum_{A_{1} \in \CA_{w}^{m, j_{1}}} \cdots \sum_{A_{r} \in \CA_{\ed(A_{r-1})}^{j_{r-1}, j_{r}}} (-1)^{|A_{1}|+\cdots+|A_{r}|-r} q^{\pair{\ve_{j}}{\down(A_{1}, \ldots, A_{r})+\xi}} \\ 
& \quad \times \sum_{B \in \CA(\ed(A_{r}), \Gamma_{j}(j))} (-1)^{|B|} \gch V_{\ed(B)t_{\down(A_{1}, \ldots, A_{r}, B)+\xi}}^{-}(\lambda + \ve_{j}), 
\end{split}
\end{equation*}
\end{cor}

The following theorem is the ``second half'' of identities of inverse Chevalley type. 
\begin{thm} \label{thm:IC_second-half}
For $w \in W$, $m = 1, \ldots, n$, and $\lambda \in P^{+}$ such that $\lambda + \ve_{k} \in P^{+}$ for all $k = 1, \ldots, n$ and $\lambda - \ve_{k} \in P^{+}$ for $k = m+1, \ldots, n$, there holds the following identity: 
\begin{equation} \label{eq:IC_second-half}
\begin{split}
& e^{-w\ve_{m}} \gch V_{w}^{-}(\lambda) \\ 
&= \sum_{B \in \CA(w, \Theta_{m})} (-1)^{|B|} \gch V_{\ed(B)t_{\down(B)}}^{-}(\lambda - \ve_{m}) \\ 
& \quad + \sum_{j = m+1}^{n} \sum_{(j_{1}, \ldots, j_{r}) \in \CS_{\overline{m}, \overline{j}}} \sum_{A_{1} \in \CA_{w}^{\overline{m}, j_{1}}} \cdots \sum_{A_{r} \in \CA_{\ed(A_{r-1})}^{j_{r-1}, j_{r}}} (-1)^{|A_{1}|+\cdots+|A_{r}|-r} q^{-\pair{\ve_{j}}{\down(A_{1}, \ldots, A_{r})}} \\ 
& \quad \times \sum_{B \in \CA(\ed(A_{r}), \Theta_{j})} (-1)^{|B|} \gch V_{\ed(B)t_{\down(A_{1}, \ldots, A_{r}, B)}}^{-}(\lambda - \ve_{j}) \\ 
& \quad + \sum_{j = 1}^{n} \sum_{(j_{1}, \ldots, j_{r}) \in \CS_{\overline{m}, j}} \sum_{A_{1} \in \CA_{w}^{\overline{m}, j_{1}}} \cdots \sum_{A_{r} \in \CA_{\ed(A_{r-1})}^{j_{r-1}, j_{r}}} (-1)^{|A_{1}|+\cdots+|A_{r}|-r} q^{\pair{\ve_{j}}{\down(A_{1}, \ldots, A_{r})}} \\ 
& \quad \times \sum_{B \in \CA(\ed(A_{r}), \Gamma_{j}(j))} (-1)^{|B|} \gch V_{\ed(B)t_{\down(A_{1}, \ldots, A_{r}, B)}}^{-}(\lambda + \ve_{j}). 
\end{split}
\end{equation}
\end{thm}

We give a proof of Theorem~\ref{thm:IC_second-half} in Section~\ref{sec:proof_IC}. Again, by Proposition~\ref{prop:gch_translation}, we obtain the following identities for an arbitrary $x \in W_{\af}$ (not only for $x \in W$). 
\begin{cor} \label{cor:IC_second-half_complete}
For $x = wt_{\xi} \in W_{\af}$ with $w \in W$ and $\xi \in Q^{\vee}$, $m = 1, \ldots, n$, and $\lambda \in P^{+}$ such that $\lambda + \ve_{k} \in P^{+}$ for all $k = 1, \ldots, n$ and $\lambda - \ve_{k} \in P^{+}$ for $k = m+1, \ldots, n$, there holds the following identity: 
\begin{equation*}
\begin{split}
& e^{-w\ve_{m}} \gch V_{x}^{-}(\lambda) \\ 
&= q^{-\pair{\ve_{m}}{\xi}} \sum_{B \in \CA(w, \Theta_{m})} (-1)^{|B|} \gch V_{\ed(B)t_{\down(B)+\xi}}^{-}(\lambda - \ve_{m}) \\ 
& \quad + \sum_{j = m+1}^{n} \sum_{(j_{1}, \ldots, j_{r}) \in \CS_{\overline{m}, \overline{j}}} \sum_{A_{1} \in \CA_{w}^{\overline{m}, j_{1}}} \cdots \sum_{A_{r} \in \CA_{\ed(A_{r-1})}^{j_{r-1}, j_{r}}} (-1)^{|A_{1}|+\cdots+|A_{r}|-r} q^{-\pair{\ve_{j}}{\down(A_{1}, \ldots, A_{r})+\xi}} \\ 
& \quad \times \sum_{B \in \CA(\ed(A_{r}), \Theta_{j})} (-1)^{|B|} \gch V_{\ed(B)t_{\down(A_{1}, \ldots, A_{r}, B)+\xi}}^{-}(\lambda - \ve_{j}) \\ 
& \quad + \sum_{j = 1}^{n} \sum_{(j_{1}, \ldots, j_{r}) \in \CS_{\overline{m}, j}} \sum_{A_{1} \in \CA_{w}^{\overline{m}, j_{1}}} \cdots \sum_{A_{r} \in \CA_{\ed(A_{r-1})}^{j_{r-1}, j_{r}}} (-1)^{|A_{1}|+\cdots+|A_{r}|-r} q^{\pair{\ve_{j}}{\down(A_{1}, \ldots, A_{r})+\xi}} \\ 
& \quad \times \sum_{B \in \CA(\ed(A_{r}), \Gamma_{j}(j))} (-1)^{|B|} \gch V_{\ed(B)t_{\down(A_{1}, \ldots, A_{r}, B)+\xi}}^{-}(\lambda + \ve_{j}). 
\end{split}
\end{equation*}
\end{cor}
Here we should mention that all the sums on the right-hand side of Theorems~\ref{thm:IC_first-half} and \ref{thm:IC_second-half}, together with Corollaries~\ref{cor:IC_first-half_complete} and \ref{cor:IC_second-half_complete}, are indeed finite sums.

\subsection{Cancellation-free identities of inverse Chevalley type in the first-half case} \label{sec:cancellation-free_statement}

We consider cancellations of terms in the first-half identities of inverse Chevalley type. Let $w \in W$. Take $l, m \in \{1, \ldots, n\}$ such that $l > m$. 
We define a directed path $\bp_{l, m}(w)$ in $\QBG(W)$ inductively as follows: 
\begin{enu}
\item if $l - m = 1$, then $\bp_{l, m}(w): w \xrightarrow{(l-1, l)} ws_{l-1}$; 
\item if $l - m > 1$, then assume that $\bp_{l', m'}(v)$ is defined for $v \in W$ and $l', m' \in \{1, \ldots, n\}$ such that $0 < l' - m' < l - m$. 
Take minimal $k \in \{m, \ldots, l-1\}$ such that $w \xrightarrow{(k, l)} ws_{(k, l)}$. Then, since $k - m < l - m$, $\bp_{k, m}(ws_{(k, l)})$ is defined. 
Let $\bp_{l, m}(w)$ be the directed path obtained as the concatenation of the edge $w \xrightarrow{(k, l)} ws_{(k, l)}$ with the directed path $\bp_{k, m}(ws_{(k, l)})$. 
\end{enu}

The following theorem gives the cancellation-free identities of inverse Chevalley type in the first-half case. 
\begin{thm} \label{thm:IC_cancellation-free_first-half}
For $w \in W$, $m = 1, \ldots, n$, and $\lambda \in P^{+}$ such that $\lambda + \ve_{k} \in P^{+}$ for all $k = 1, \ldots, m$, there holds the following cancellation-free identity: 
\begin{equation*}
\begin{split}
& e^{w\ve_{m}} \gch V_{w}^{-}(\lambda) \\ 
&= \sum_{B \in \CA(w, \Gamma_{m}(m))} (-1)^{|B|} \gch V_{\ed(B)t_{\down(B)}}^{-}(\lambda + \ve_{m}) \\ 
& \quad + \sum_{j = 1}^{m-1} q^{\pair{\ve_{j}}{\wt(\bp_{m, j}(w))}} \sum_{\CA(\ed(\bp_{m, j}(w)), \Gamma_{j}(j))} (-1)^{|B|} \gch V_{\ed(B)t_{\down(B) + \wt(\bp_{m, j}(w))}}^{-}(\lambda + \ve_{j}). 
\end{split}
\end{equation*}
\end{thm}
We give a proof of this theorem in Section~\ref{sec:cancellation-free_proof}. 
Again, by using Proposition~\ref{prop:gch_translation}, we obtain the following cancellation-free identities for an arbitrary $x \in W_{\af}$ (not only for $x \in W$). 
\begin{cor} \label{cor:IC_cancellation-free_first-half_complete}
For $x = wt_{\xi} \in W_{\af}$ with $w \in W$ and $\xi \in Q^{\vee}$, $m = 1, \ldots, n$, and $\lambda \in P^{+}$ such that $\lambda + \ve_{k} \in P^{+}$ for all $k = 1, \ldots, m$, there holds the following cancellation-free identity: 
\begin{equation*}
\begin{split}
& e^{w\ve_{m}} \gch V_{x}^{-}(\lambda) \\ 
&= q^{\pair{\ve_{m}}{\xi}} \sum_{B \in \CA(w, \Gamma_{m}(m))} (-1)^{|B|} \gch V_{\ed(B)t_{\down(B)+\xi}}^{-}(\lambda + \ve_{m}) \\ 
& \quad + \sum_{j = 1}^{m-1} q^{\pair{\ve_{j}}{\wt(\bp_{m, j}(w))+\xi}} \sum_{\CA(\ed(\bp_{m, j}(w)), \Gamma_{j}(j))} (-1)^{|B|} \gch V_{\ed(B)t_{\down(B) + \wt(\bp_{m, j}(w)) + \xi}}^{-}(\lambda + \ve_{j}). 
\end{split}
\end{equation*}
\end{cor}

\subsection{Conjectural cancellation-free identities of inverse Chevalley type in the second-half case} \label{sec:cancellation-free_conjecture}

We make a conjecture for the cancellations of terms in the second-half identities of inverse Chevalley type. 
We define a total order $<$ on $[\overline{n}]$ by: $1 < 2 < \cdots < n < \overline{n} < \overline{n-1} < \cdots < \overline{1}$, 
and a distance function $d(\cdot, \cdot)$ on $[\overline{n}]$ as follows: 
\begin{enu}
\item for $k \in [\overline{n}]$, set $d(k, k) := 0$; 
\item for $k, l \in [\overline{n}]$ with $k > l$, set 
\begin{equation*}
d(k, l) := 
\begin{cases} 
k - l & \text{if $1 \le l < k \le n$,} \\ 
(2n+1-p) - l & \text{if $l = \overline{p}$ for some $1 \le p \le n$ and $1 \le l \le n$,} \\ 
q - p & \text{if $k = \overline{p}$ and $l = \overline{q}$ for some $1 \le p \le q \le n$;} 
\end{cases} 
\end{equation*} 
\item for $k, l \in [\overline{n}]$ with $k < l$, set $d(k, l) := d(l, k)$. 
\end{enu}
Also, for $1 \le l \le n$ and $1 \le k < \overline{l}$, we define $\gamma_{\overline{l}, k} \in \Delta^{+}$ by 
\begin{equation*}
\gamma_{\overline{l}, k} := 
\begin{cases}
(k, \overline{l}) & \text{if $k = 1, \ldots, l$}, \\ 
(l, \overline{k}) & \text{if $k = l+1, \ldots, n$}, \\ 
(l, p) & \text{if $k = \overline{p}$ for some $p = l+1, \ldots, n$}. 
\end{cases}
\end{equation*} 

Now, we define a directed path $\bp_{\overline{l}, m}(w)$ in $\QBG(W)$ for $w \in W$, $1 \le l \le n$, and $1 \le m \le \overline{l+1}$
by induction on $d(\overline{l}, m)$ as follows; 
we understand that $\overline{n+1} := n$.
\begin{enu}
\item If $d(\overline{l}, m) = 1$ (in this case, $\gamma_{\overline{l}, m}$ is a simple root), then $\bp_{\overline{l}, m}(w) : w \xrightarrow{\gamma_{\overline{l}, m}} ws_{\gamma_{\overline{l}, m}}$. 
\item If $d(\overline{l}, m) > 1$, then assume that $\bp_{q, m'}(v)$ is defined for $v \in W$ and $q, m' \in [\overline{n}]$, with $q > m'$, such that $0 < d(q, m') < d(\overline{l}, m)$; 
note that for $1 \le q \le n$, $\bp_{q, m'}(v)$ is already defined in Section~\ref{sec:cancellation-free_statement}. 
Take minimal $k \in [\overline{n}]$, with $m \le k \le \overline{l+1}$, such that $w \xrightarrow{\gamma_{\overline{l}, k}} ws_{\gamma_{\overline{l}, k}}$. 
Then, since $0 < d(k, m) < d(\overline{l}, m)$, $\bp_{k, m}(ws_{\gamma_{\overline{l}, k}})$ is defined. 
Let $\bp_{\overline{l}, m}(w)$ be the directed path obtained as the concatenation of the edge $w \xrightarrow{\gamma_{\overline{l}, k}} ws_{\gamma_{\overline{l}, k}}$ with the directed path $\bp_{k, m}(ws_{\gamma_{\overline{l}, k}})$. 
\end{enu}

Our conjectural cancellation-free identities in the second-half case are as follows. 
\begin{conj} \label{conj:IC_cancellation_second-half}
For $w \in W$, $m = 1, \ldots, n$, and $\lambda \in P^{+}$ such that $\lambda + \ve_{k} \in P^{+}$ for all $k = 1, \ldots, n$ and such that $\lambda - \ve_{k} \in P^{+}$ for $k = m+1, \ldots, n$, there exists some $m \le l \le n$ for which the following cancellation-free identity holds: 
\begin{align*}
& e^{-w\ve_{m}} \gch V_{w}^{-}(\lambda) \\ 
&= \sum_{B \in \CA(w, \Theta_{m})} (-1)^{|B|} \gch V_{\ed(B)t_{\down(B)}}^{-}(\lambda - \ve_{m}) \\ 
& \quad + \sum_{k = m+1}^{n} q^{-\pair{\ve_{k}}{\down(\bp_{\overline{m}, \overline{k}}(w))}} \\ 
& \quad \times \sum_{B \in \CA(\ed(\bp_{\overline{m}, \overline{k}}(w)), \Theta_{k})} (-1)^{|B|} \gch V_{\ed(B)t_{\down(B)+\down(\bp_{\overline{m}, \overline{k}}(w))}}^{-}(\lambda - \ve_{k}) \\ 
& \quad + \sum_{k = 1}^{l} q^{\pair{\ve_{k}}{\down(\bp_{\overline{m}, k}(w))}} \\ 
& \quad \times \sum_{B \in \CA(\ed(\bp_{\overline{m}, k}(w)), \Gamma_{k}(k))} (-1)^{|B|} \gch V_{\ed(B)t_{\down(B)+\down(\bp_{\overline{m}, k}(w))}}^{-}(\lambda + \ve_{k}). 
\end{align*}
\end{conj}

\begin{conj}
For $x = wt_{\xi} \in W_{\af}$ with $w \in W$ and $\xi \in Q^{\vee}$, $m = 1, \ldots, n$, and $\lambda \in P^{+}$ such that $\lambda + \ve_{k} \in P^{+}$ for all $k = 1, \ldots, n$ and such that $\lambda - \ve_{k} \in P^{+}$ for $k = m+1, \ldots, n$, there exists some $m \le l \le n$ for which the following cancellation-free identity holds: 
\begin{align*}
& e^{-w\ve_{m}} \gch V_{x}^{-}(\lambda) \\ 
&= q^{-\pair{\ve_{m}}{\xi}} \sum_{B \in \CA(w, \Theta_{m})} (-1)^{|B|} \gch V_{\ed(B)t_{\down(B)+\xi}}^{-}(\lambda - \ve_{m}) \\ 
& \quad + \sum_{k = m+1}^{n} q^{-\pair{\ve_{k}}{\down(\bp_{\overline{m}, \overline{k}}(w))+\xi}} \\ 
& \quad \times \sum_{B \in \CA(\ed(\bp_{\overline{m}, \overline{k}}(w)), \Theta_{k})} (-1)^{|B|} \gch V_{\ed(B)t_{\down(B)+\down(\bp_{\overline{m}, \overline{k}}(w))+\xi}}^{-}(\lambda - \ve_{k}) \\ 
& \quad + \sum_{k = 1}^{l} q^{\pair{\ve_{k}}{\down(\bp_{\overline{m}, k}(w))+\xi}} \\ 
& \quad \times \sum_{B \in \CA(\ed(\bp_{\overline{m}, k}(w)), \Gamma_{k}(k))} (-1)^{|B|} \gch V_{\ed(B)t_{\down(B)+\down(\bp_{\overline{m}, k}(w))+\xi}}^{-}(\lambda + \ve_{k}). 
\end{align*}
\end{conj}

We expect that the $l$ in the above conjectures is either $m$ or $n$. 

\subsection{Examples of identities of inverse Chevalley type}

We give an example of the identities given by Theorem~\ref{thm:IC_cancellation-free_first-half}, and also give two examples of the conjectural identities proposed by Conjecture~\ref{conj:IC_cancellation_second-half}; 
in these examples, we assume that $n = 3$. 
We used SageMath \cite{Sage} for calculations in these examples. 

\begin{exm}
We consider the product $e^{w\ve_{3}} \gch V_{w}^{-}(\lambda)$. Let $w = s_{1}s_{2}s_{1}$, and take $\lambda \in P^{+}$ such that $\lambda + \ve_{k} \in P^{+}$ for $k = 1, 2, 3$. 
We see that 
\begin{equation*}
\CS_{3, 1} = \{ (1), (2, 1) \}, \quad \CS_{3, 2} = \{ (2) \}, 
\end{equation*}
and that the admissible subsets which appear on the right-hand side of \eqref{eq:IC_first-half} are as follows: 
\begin{equation*}
\CA_{s_{1}s_{2}s_{1}}^{3, 1} = \{ \underbrace{\{1\}}_{=: A_{1}}, \underbrace{\{1, 2\}}_{=: A_{2}} \}, \quad \CA_{s_{1}s_{2}s_{1}}^{3, 2} = \{ \underbrace{\{2\}}_{=: A_{3}} \}, \quad \CA_{s_{2}s_{1}}^{2, 1} = \{ \underbrace{\{1\}}_{=: A_{4}} \}. 
\end{equation*}
Table~\ref{tab:list_statistics1} is the list of $\ed(\cdot)$ and $\down(\cdot)$ for admisible subsets $A_{1}, A_{2}, A_{3}, A_{4}$. 
\begin{table}[ht]
\centering
\caption{The list of $\ed(A)$ and $\down(A)$ for $A = A_{1}, A_{2}, A_{3}, A_{4}$}
\label{tab:list_statistics1}
\begin{tabular}{|c||c|c|} \hline
Admissible subsets & $\ed(\cdot)$ & $\down(\cdot)$ \\ \hline
$A_{1}$ & $e$ & $\alpha_{1}^{\vee}+\alpha_{2}^{\vee}$ \\ 
$A_{2}$ & $s_{2}$ & $\alpha_{1}^{\vee}+\alpha_{2}^{\vee}$ \\ 
$A_{3}$ & $s_{2}s_{1}$ & $\alpha_{2}^{\vee}$ \\ 
$A_{4}$ & $s_{2}$ & $\alpha_{1}^{\vee}$ \\ \hline
\end{tabular}
\end{table}

Therefore, by Theorem~\ref{thm:IC_first-half}, we compute: 
\begin{align*}
& e^{s_{1}s_{2}s_{1}\ve_{3}} \gch V_{s_{1}s_{2}s_{1}}^{-}(\lambda) \\ 
\begin{split}
&= \sum_{B \in \CA(s_{1}s_{2}s_{1}, \Gamma_{3}(3))} (-1)^{|B|} \gch V_{\ed(B)t_{\down(B)}}^{-}(\lambda + \ve_{3}) \\ 
& \quad \underbrace{{}+ q^{\pair{\ve_{2}}{\alpha_{2}^{\vee}}} \sum_{B \in \CA(s_{2}s_{1}, \Gamma_{2}(2))} (-1)^{|B|} \gch V_{\ed(B)t_{\down(B)+\alpha_{2}^{\vee}}}^{-}(\lambda + \ve_{2})}_{A_{3}} \\ 
& \quad \underbrace{{}+ q^{\pair{\ve_{1}}{\alpha_{1}^{\vee}+\alpha_{2}^{\vee}}} \sum_{B \in \CA(e, \Gamma_{1}(1))} (-1)^{|B|} \gch V_{\ed(B)t_{\down(B)+\alpha_{1}^{\vee}+\alpha_{2}^{\vee}}}^{-}(\lambda + \ve_{1})}_{A_{1}} \\ 
& \quad \underbrace{{}- q^{\pair{\ve_{1}}{\alpha_{1}^{\vee}+\alpha_{2}^{\vee}}} \sum_{B \in \CA(s_{2}, \Gamma_{1}(1))} (-1)^{|B|} \gch V_{\ed(B)t_{\down(B)+\alpha_{1}^{\vee}+\alpha_{2}^{\vee}}}^{-}(\lambda + \ve_{1})}_{A_{2}} \\ 
& \quad \underbrace{{}+q^{\pair{\ve_{1}}{\alpha_{1}^{\vee}+\alpha_{2}^{\vee}}} \sum_{B \in \CA(s_{2}, \Gamma_{1}(1))} (-1)^{|B|} \gch V_{\ed(B)t_{\down(B)+\alpha_{1}^{\vee}+\alpha_{2}^{\vee}}}^{-}(\lambda + \ve_{1})}_{(A_{3}, A_{4})}
\end{split} \\ 
\begin{split}
&= \sum_{B \in \CA(s_{1}s_{2}s_{1}, \Gamma_{3}(3))} (-1)^{|B|} \gch V_{\ed(B)t_{\down(B)}}^{-}(\lambda + \ve_{3}) \\ 
& \quad + q^{\pair{\ve_{2}}{\alpha_{2}^{\vee}}} \sum_{B \in \CA(s_{2}s_{1}, \Gamma_{2}(2))} (-1)^{|B|} \gch V_{\ed(B)t_{\down(B)+\alpha_{2}^{\vee}}}^{-}(\lambda + \ve_{2}) \\ 
& \quad + q^{\pair{\ve_{1}}{\alpha_{1}^{\vee}+\alpha_{2}^{\vee}}} \sum_{B \in \CA(e, \Gamma_{1}(1))} (-1)^{|B|} \gch V_{\ed(B)t_{\down(B)+\alpha_{1}^{\vee}+\alpha_{2}^{\vee}}}^{-}(\lambda + \ve_{1}). \\ 
\end{split}
\end{align*}
The above result agrees with Theorem~\ref{thm:IC_cancellation-free_first-half},
since we have 
\begin{equation*}
\bp_{3, 2}(s_{1}s_{2}s_{1}): s_{1}s_{2}s_{1} \xrightarrow{(2, 3)} s_{2}s_{1}, \quad \bp_{3, 1}(s_{1}s_{2}s_{1}): s_{1}s_{2}s_{1} \xrightarrow{(1, 3)} e. 
\end{equation*}
\end{exm}

\begin{exm}
We consider the product $e^{-w\ve_{2}}\gch V_{w}^{-}(\lambda)$. Let $w = s_{3}s_{2}$, and take $\lambda \in P^{+}$ such that $\lambda + \ve_{k} \in P^{+}$ for $k = 1, 2, 3$ and such that $\lambda - \ve_{k} \in P^{+}$ for $k = 2, 3$. 
We see that the sets $\CS_{\overline{2}, k}$ for $k = 1, 2, 3, \overline{3}$ 
are: 
\begin{align*}
\CS_{\overline{2}, 1} &= \{ (1), (2, 1), (3, 1), (\overline{3}, 1), (3, 2, 1), (\overline{3}, 2, 1), (\overline{3}, 3, 1), (\overline{3}, 3, 2, 1) \}, \\ 
\CS_{\overline{2}, 2} &= \{ (2), (3, 2), (\overline{3}, 2), (\overline{3}, 3, 2) \}, \\ 
\CS_{\overline{2}, 3} &= \{ (3), (\overline{3}, 3) \}, \\ 
\CS_{\overline{2}, \overline{3}} &= \{ (\overline{3}) \}, 
\end{align*}
and that the admissible subsets which appear on the right-hand side of \eqref{eq:IC_second-half} are as follows: 
\begin{align*}
\CA_{s_{3}s_{2}}^{\overline{2}, 1} &= \emptyset, & \CA_{s_{3}s_{2}}^{\overline{2}, 2} &= \{ \underbrace{\{2, 4\}}_{=: A_{1}}, \underbrace{\{2, 3, 4\}}_{=: A_{2}} \}, & \CA_{s_{3}s_{2}}^{\overline{2}, 3} &= \{ \underbrace{\{2\}}_{=: A_{3}}, \underbrace{\{2, 3\}}_{=: A_{4}} \}, & \CA_{s_{3}s_{2}}^{\overline{2}, \overline{3}} &= \{ \underbrace{\{4\}}_{=: A_{5}} \}, \\ 
\CA_{s_{3}}^{\overline{3}, 1} &= \emptyset, & \CA_{s_{3}}^{\overline{3}, 2} &= \{ \underbrace{\{2\}}_{=: A_{6}}, \underbrace{\{2, 3\}}_{=: A_{7}} \}, & \CA_{s_{3}}^{\overline{3}, 3} &= \{ \underbrace{\{3\}}_{=: A_{8}} \}, \\ 
\CA_{e}^{3, 1} &= \emptyset, & \CA_{e}^{3, 2} &= \{ \underbrace{\{2\}}_{=: A_{9}} \}, \\
\CA_{s_{2}s_{3}s_{2}}^{3, 1} &= \{ \underbrace{\{1\}}_{=: A_{10}}, \underbrace{\{1, 2\}}_{=: A_{11}} \}, & \CA_{s_{2}s_{3}s_{2}}^{3, 2} &= \{ \underbrace{\{2\}}_{=: A_{12}} \}, \\ 
\CA_{s_{2}s_{3}}^{2, 1} &= \{ \underbrace{\{1\}}_{=: A_{13}} \}, \\ 
\CA_{s_{2}}^{2, 1} &= \{ \underbrace{\{1\}}_{=: A_{14}} \}. 
\end{align*}
Table~\ref{tab:list_statistics2} is the list of $\ed(\cdot)$ and $\down(\cdot)$ for admissible subsets given above. 
\begin{table}[ht]
\centering
\caption{The list of $\ed(A)$ and $\down(A)$ for $A = A_{1}, \ldots, A_{14}$}
\label{tab:list_statistics2}
\begin{tabular}{|c||c|c|} \hline
Admissible subset & $\ed(\cdot)$ & $\down(\cdot)$ \\ \hline
$A_{1}$ & $s_{2}s_{3}$ & $\alpha_{2}^{\vee}$ \\ 
$A_{2}$ & $s_{2}$ & $\alpha_{2}^{\vee}+\alpha_{3}^{\vee}$ \\ 
$A_{3}$ & $s_{2}s_{3}s_{2}$ & $0$ \\ 
$A_{4}$ & $e$ & $\alpha_{2}^{\vee}+\alpha_{3}^{\vee}$ \\ 
$A_{5}$ & $s_{3}$ & $\alpha_{2}^{\vee}$ \\ 
$A_{6}$ & $s_{2}s_{3}$ & $0$ \\ 
$A_{7}$ & $s_{2}$ & $\alpha_{3}^{\vee}$ \\ 
$A_{8}$ & $e$ & $\alpha_{3}^{\vee}$ \\ 
$A_{9}$ & $s_{2}$ & $0$ \\ 
$A_{10}$ & $s_{2}s_{3}s_{1}s_{2}$ & $0$ \\ 
$A_{11}$ & $s_{2}s_{3}s_{1}$ & $\alpha_{2}^{\vee}$ \\ 
$A_{12}$ & $s_{2}s_{3}$ & $\alpha_{2}^{\vee}$ \\ 
$A_{13}$ & $s_{2}s_{3}s_{1}$ & $0$ \\ 
$A_{14}$ & $s_{2}s_{1}$ & $0$ \\ \hline
\end{tabular}
\end{table}

Therefore, by Theorem~\ref{thm:IC_second-half}, we see that 
\begin{align*}
& e^{-s_{3}s_{2}\ve_{2}} \gch V_{s_{3}s_{2}}^{-}(\lambda) \\ 
\begin{split}
&= \sum_{B \in \CA(s_{3}s_{2}, \Theta_{2})} (-1)^{|B|} \gch V_{\ed(B)t_{\down(B)}}^{-}(\lambda - \ve_{2}) \\ 
& \quad \underbrace{{}+{} q^{-\pair{\ve_{3}}{\alpha_{2}^{\vee}}} \sum_{B \in \CA(s_{3}, \Theta_{3})} (-1)^{|B|} \gch V_{\ed(B)t_{\down(B)+\alpha_{2}^{\vee}}}^{-}(\lambda - \ve_{3})}_{A_{5}} \\ 
& \quad \underbrace{{}+{} \sum_{B \in \CA(s_{2}s_{3}s_{2}, \Gamma_{3}(3))} (-1)^{|B|} \gch V_{\ed(B)t_{\down(B)}}^{-}(\lambda + \ve_{3})}_{A_{3}} \\ 
& \quad \underbrace{{}+{} q^{\pair{\ve_{2}}{\alpha_{2}^{\vee}}} \sum_{B \in \CA(s_{2}s_{3}, \Gamma_{2}(2))} (-1)^{|B|} \gch V_{\ed(B)t_{\down(B)+\alpha_{2}^{\vee}}}^{-}(\lambda + \ve_{2})}_{(A_{3}, A_{12})} \\ 
& \quad \underbrace{{}+{} \sum_{B \in \CA(s_{2}s_{3}s_{1}s_{2}, \Gamma_{1}(1))} (-1)^{|B|} \gch V_{\ed(B)t_{\down(B)}}^{-}(\lambda + \ve_{1})}_{(A_{3}, A_{10})}; 
\end{split}
\end{align*}
here, the other terms cancel out. 
The above result agrees with Conjecture~\ref{conj:IC_cancellation_second-half} if we take $l = 3$, since we have 
\begin{align*}
& \bp_{\overline{2}, \overline{3}}(s_{3}s_{2}): s_{3}s_{2} \xrightarrow{(2, 3)} s_{3}, & & \bp_{\overline{2}, 3}(s_{3}s_{2}): s_{3}s_{2} \xrightarrow{(2, \overline{3})} s_{2}s_{3}s_{2}, \\
& \bp_{\overline{2}, 2}(s_{3}s_{2}): s_{3}s_{2} \xrightarrow{(2, \overline{3})} s_{2}s_{3}s_{2} \xrightarrow{(2, 3)} s_{2}s_{3}, & & \bp_{\overline{2}, 1}(s_{3}s_{2}): s_{3}s_{2} \xrightarrow{(2, \overline{3})} s_{2}s_{3}s_{2} \xrightarrow{(1, 3)} s_{2}s_{3}s_{1}s_{2}. 
\end{align*}
\end{exm}

\begin{exm} 
We consider the product $e^{-w\ve_{1}}\gch V_{w}^{-}(\lambda)$. Let $w = s_{1}s_{2}s_{3}s_{2}s_{1}$, and take $\lambda \in P^{+}$ such that $\lambda + \ve_{k} \in P^{+}$ for $k = 1, 2, 3$ and such that $\lambda - \ve_{k} \in P^{+}$ for $k = 1, 2, 3$. 
We see that the sets $\CS_{\overline{1}, k}$ for $k = 1, 2, 3, \overline{3}, \overline{2}$ are: 
\begin{align*}
\CS_{\overline{1}, 1} &= \{ (1), (2, 1), (3, 1), (\overline{3}, 1), (\overline{2}, 1), (3, 2, 1), (\overline{3}, 2, 1), (\overline{2}, 2, 1), (\overline{3}, 3, 1), (\overline{2}, 3, 1), (\overline{2}, \overline{3}, 1), \\ 
& \qquad (\overline{3}, 3, 2, 1), (\overline{2}, 3, 2, 1), (\overline{2}, \overline{3}, 2, 1), (\overline{2}, \overline{3}, 3, 1), (\overline{2}, \overline{3}, 3, 2, 1) \}, \\
\CS_{\overline{1}, 2} &= \{ (2), (3, 2), (\overline{3}, 2), (\overline{2}, 2), (\overline{3}, 3, 2), (\overline{2}, 3, 2), (\overline{2}, \overline{3}, 2), (\overline{2}, \overline{3}, 3, 2) \} \\ 
\CS_{\overline{1}, 3} &= \{ (3), (\overline{3}, 3), (\overline{2}, 3), (\overline{2}, \overline{3}, 3) \} \\ 
\CS_{\overline{1}, \overline{3}} &= \{ (\overline{3}), (\overline{2}, \overline{3}) \} \\ 
\CS_{\overline{1}, \overline{2}} &= \{ (\overline{2}) \}, 
\end{align*}
and that the admissible subsets which appear on the right-hand side of \eqref{eq:IC_second-half} are as follows: 
\begin{align*}
\CA_{s_{1}s_{2}s_{3}s_{2}s_{1}}^{\overline{1}, 1} &= \{\underbrace{\{3\}}_{=: A_{1}}\}, & \CA_{s_{1}s_{2}s_{3}s_{2}s_{1}}^{\overline{1}, 2} &= \{\underbrace{\{3, 5\}}_{=:A_{2}}\}, & \CA_{s_{1}s_{2}s_{3}s_{2}s_{1}}^{\overline{1}, 3} &= \emptyset, \\ \CA_{s_{1}s_{2}s_{3}s_{2}s_{1}}^{\overline{1}, \overline{3}} &= \emptyset, & \CA_{s_{1}s_{2}s_{3}s_{2}s_{1}}^{\overline{1}, \overline{2}} &= \{\underbrace{\{5\}}_{=:A_{3}}\}, \\ 
\CA_{s_{1}s_{2}s_{3}s_{2}}^{\overline{2}, 1} &= \emptyset, & \CA_{s_{1}s_{2}s_{3}s_{2}}^{\overline{2}, 2} &= \{\underbrace{\{3\}}_{=:A_{4}}\}, & \CA_{s_{1}s_{2}s_{3}s_{2}}^{\overline{2}, 3} &= \{\underbrace{\{3, 4\}}_{=:A_{5}}\}, \\ 
\CA_{s_{1}s_{2}s_{3}s_{2}}^{\overline{2}, \overline{3}} &= \{\underbrace{\{4\}}_{=:A_{6}}\}, \\ 
\CA_{s_{1}s_{2}s_{3}}^{\overline{3}, 1} &= \emptyset, & \CA_{s_{1}s_{2}s_{3}}^{\overline{3}, 2} &= \emptyset, & \CA_{s_{1}s_{2}s_{3}}^{\overline{3}, 3} &= \{\underbrace{\{3\}}_{=:A_{7}}\}, \displaybreak[1] \\ 
\CA_{s_{1}s_{2}}^{3, 1} &= \emptyset, & \CA_{s_{1}s_{2}}^{3, 2} &= \{\underbrace{\{2\}}_{=:A_{8}}\}, \\ 
\CA_{s_{1}}^{2, 1} &= \{\underbrace{\{1\}}_{=:A_{9}}\}. 
\end{align*}
Table~\ref{tab:list_statistics3} is the list of $\ed(\cdot)$ and $\down(\cdot)$ for admissible subsets given above. 
\begin{table}[ht]
\centering
\caption{The list of $\ed(A)$ and $\down(A)$ for $A = A_{1}, \ldots, A_{9}$}
\label{tab:list_statistics3}
\begin{tabular}{|c||cc|} \hline 
$A$ & $\ed(A)$ & $\down(A)$ \\ \hline 
$A_{1}$ & $e$ & $\alpha_{1}^{\vee}+\alpha_{2}^{\vee}+\alpha_{3}^{\vee}$ \\ 
$A_{2}$ & $s_{1}$ & $\alpha_{1}^{\vee}+\alpha_{2}^{\vee}+\alpha_{3}^{\vee}$ \\ 
$A_{3}$ & $s_{1}s_{2}s_{3}s_{2}$ & $\alpha_{1}^{\vee}$ \\ 
$A_{4}$ & $s_{1}$ & $\alpha_{2}^{\vee}+\alpha_{3}^{\vee}$ \\ 
$A_{5}$ & $s_{1}s_{2}$ & $\alpha_{2}^{\vee}+\alpha_{3}^{\vee}$ \\ 
$A_{6}$ & $s_{1}s_{2}s_{3}$ & $\alpha_{2}^{\vee}$ \\ 
$A_{7}$ & $s_{1}s_{2}$ & $\alpha_{3}^{\vee}$ \\ 
$A_{8}$ & $s_{1}$ & $\alpha_{2}^{\vee}$ \\ 
$A_{9}$ & $e$ & $\alpha_{1}^{\vee}$ \\ \hline
\end{tabular}
\end{table}

Therefore, by Theorem~\ref{thm:IC_second-half}, we see that 
\begin{align*}
& e^{-s_{1}s_{2}s_{3}s_{2}s_{1}\ve_{1}} \gch V_{s_{1}s{2}s_{3}s_{2}s_{1}}^{-}(\lambda) \\ 
&= \sum_{B \in \CA(s_{1}s_{2}s_{3}s_{2}s_{1}, \Theta_{1})} (-1)^{|B|} \gch V_{\ed(B)t_{\down(B)}}^{-}(\lambda - \ve_{1}) \\ 
& \quad + \underbrace{q^{-\pair{\ve_{2}}{\alpha_{1}^{\vee}}} \sum_{B \in \CA(s_{1}s_{2}s_{3}s_{2}, \Theta_{2})} (-1)^{|B|} \gch V_{\ed(B)t_{\down(B)+\alpha_{1}^{\vee}}}^{-}(\lambda - \ve_{2})}_{A_{3}} \\ 
& \quad + \underbrace{q^{-\pair{\ve_{3}}{\alpha_{1}^{\vee}+\alpha_{2}^{\vee}}} \sum_{B \in \CA(s_{1}s_{2}s_{3}, \Theta_{3})} (-1)^{|B|} \gch V_{\ed(B)t_{\down(B)+\alpha_{1}^{\vee}+\alpha_{2}^{\vee}}}^{-} (\lambda - \ve_{3})}_{(A_{3}, A_{6})} \\ 
& \quad + \underbrace{q^{\pair{\ve_{1}}{\alpha_{1}^{\vee}+\alpha_{2}^{\vee}+\alpha_{3}^{\vee}}} \sum_{B \in \CA(e, \Gamma_{1}(1))} (-1)^{|B|} \gch V_{\ed(B)t_{\down(B)+\alpha_{1}^{\vee}+\alpha_{2}^{\vee}+\alpha_{3}^{\vee}}}^{-} (\lambda + \ve_{1})}_{A_{1}}; 
\end{align*}
here, the other terms cancel out. 
The above result agrees with Conjecture~\ref{conj:IC_cancellation_second-half} if we take $l = 1$, since we have 
\begin{align*}
& \bp_{\overline{1}, \overline{2}}(s_{1}s_{2}s_{3}s_{2}s_{1}): s_{1}s_{2}s_{3}s_{2}s_{1} \xrightarrow{(1, 2)} s_{1}s_{2}s_{3}s_{2}, \\
& \bp_{\overline{1}, \overline{3}}(s_{1}s_{2}s_{3}s_{2}s_{1}): s_{1}s_{2}s_{3}s_{2}s_{1} \xrightarrow{(1, 2)} s_{1}s_{2}s_{3}s_{2} \xrightarrow{(2, 3)} s_{1}s_{2}s_{3}, \\ 
& \bp_{\overline{1}, 1}(s_{1}s_{2}s_{3}s_{2}s_{1}): s_{1}s_{2}s_{3}s_{2}s_{1} \xrightarrow{(1, \overline{1})} e. 
\end{align*}
\end{exm}

\section{Proofs of Theorems~\ref{thm:IC_first-half} and \ref{thm:IC_second-half}} \label{sec:proof_IC}
We give proofs of our identities of inverse Chevalley type. 

\subsection{First-half case}
First, we consider the first-half case. The following proposition is a key to the proof of the identities. 
\begin{prop} \label{prop:key_first-half}
Let $w \in W$ and $\lambda \in P^{+}$ such that $\lambda + \ve_{k} \in P^{+}$. Then we have 
\begin{equation} \label{eq:identity_C-IC}
\begin{split}
& \sum_{B \in \CA(w, \Gamma_{k}(k))} (-1)^{|B|} \gch V_{\ed(B)t_{\down(B)}}^{-}(\lambda + \ve_{k}) \\ 
&= \sum_{A \in \CA(w, \Theta_{k})} (-1)^{|A|} e^{w\ve_{k}} \gch V_{\ed(A)t_{\down(A)}}^{-}(\lambda). 
\end{split}
\end{equation}
\end{prop}
\begin{proof}
In this proof, 
\begin{enu}
\item the sequence $\Gamma_{k}(k)$ is of the form: 
\begin{equation*}
\Gamma_{k}(k) = (\beta_{1, \overline{k}}, \ldots, \beta_{k-1, \overline{k}}, \beta_{k, \overline{k+1}}, \ldots, \beta_{k, \overline{n}}, \beta_{k, \overline{k}}, \beta_{k, n}, \ldots, \beta_{k, k+1}), 
\end{equation*}
\item admissible subsets $B \in \CA(v, \Gamma_{k}(k))$ for $v \in W$ are subsets of the (totally ordered) index set 
\begin{equation*}
I_{1} := \{(1, \overline{k}) \vtl \cdots \vtl (k-1, \overline{k}) \vtl (k, \overline{k+1}) \vtl \cdots \vtl (k, \overline{n}) \vtl (k, \overline{k}) \vtl (k, n) \vtl \cdots \vtl (k, k+1)\}; 
\end{equation*}
here, $\vtl$ defines a total order. 
\end{enu}
Similarly, 
\begin{enu}
\item the sequence $\Gamma_{k-1, k}$ is of the form: 
\begin{equation*}
\Gamma_{k-1, k} = (\gamma_{k, k+1}, \ldots, \gamma_{k, n}, \gamma_{k, \overline{k}}, \gamma_{k, \overline{n}}, \ldots, \gamma_{k, \overline{k+1}}, \gamma_{k-1, \overline{k}}, \ldots, \gamma_{1, \overline{k}}, \gamma_{1, k}, \ldots, \gamma_{k-1, k}), 
\end{equation*}
\item admissible subsets $A \in \Gamma(v, \Gamma_{k-1, k})$ for $v \in W$ are subsets of the (totally ordered) index set 
\begin{equation*}
\begin{split}
& I_{2} := \{(k, k+1) \prec \cdots \prec (k, n) \prec (k, \overline{k}) \prec (k, \overline{n}) \prec \cdots \prec (k, \overline{k+1}) \\ 
& \hspace*{60mm} \prec (k-1, \overline{k}) \prec \cdots \prec (1, \overline{k}) \prec (1, k) \prec \cdots \prec (k-1, k)\}; 
\end{split}
\end{equation*}
here, $\prec$ defines a total order. Note that $I_{1} \subset I_{2}$, and that if $\beta \vtl \gamma$, then $\gamma \prec \beta$ for $\beta, \gamma \in I_{1}$.\end{enu}

By Proposition~\ref{prop:Chevalley_ek} and Lemma~\ref{lem:ek-chain}, we deduce that 
\begin{align}
& \text{(LHS of \eqref{eq:identity_C-IC})} \nonumber \\
&= \sum_{B \in \CA(w, \Gamma_{k}(k))} (-1)^{|B|} \gch V_{\ed(B)t_{\down(B)}}^{-}(\lambda + \ve_{k}) \nonumber \\
&= \sum_{B \in \CA(w, \Gamma_{k}(k))} (-1)^{|B|} q^{-\pair{\lambda + \ve_{k}}{\down(B)}} \gch V_{\ed(B)}^{-}(\lambda + \ve_{k}) \nonumber \\ 
\begin{split}
&= \sum_{B \in \CA(w, \Gamma_{k}(k))} (-1)^{|B|} q^{-\pair{\lambda + \ve_{k}}{\down(B)}} \\ 
& \hspace*{10mm} \times \frac{1}{1-q^{-\pair{\lambda+\vpi_{k}}{\alpha_{k}^{\vee}}}} \sum_{A \in \CA(\ed(B), \Gamma_{k-1, k})} (-1)^{n(A)} q^{-\height(A)} e^{\wt(A)} \gch V_{\ed(A)t_{\down(A)}}^{-}(\lambda). 
\end{split} \label{eq:C-IC_1}
\end{align}
Let us take the alcove path
\begin{equation*}
(A_{\circ} = \underbrace{A_{0}, A_{1}, \ldots, A_{a}}_{\Gamma_{k}^{\ast}(k)} = \underbrace{B_{0}, B_{1}, \ldots, B_{b}}_{\Theta_{k}} = A_{-\ve_{k}})
\end{equation*}
corresponding to $\Gamma_{k-1, k} = \Gamma_{k}^{\ast}(k) \ast \Theta_{k}$. 
Then, the hyperplane containing the common wall of $A_{t-1}$ and $A_{t}$, $t = 1, \ldots, a$, is of the form $H_{\beta, 0}$ with $\beta \in \Delta^{+}$. 
Also, the hyperplane containing the common wall of $B_{t-1}$ and $B_{t}$ ($t = 1, \ldots, b$) is of the form $H_{\beta, 1}$ with $\beta \in \Delta^{+}$ (see Remark~\ref{rem:common_wall}). 
This implies that if we divide $A \in \CA(v, \Gamma_{k-1, k})$ into the two parts: 
$A^{(1)} := A \cap \{ (k, k+1), \ldots, (1, \overline{k}) \} \ (\in \CA(v, \Gamma_{k}^{\ast}(k)))$ and $A^{(2)} := A \cap \{ (1, k), \ldots, (k-1, k) \} \ (\in \CA(\ed(A^{(1)}), \Theta_{k}))$, 
then we have 
\begin{align*}
\height(A) &= \sum_{a \in A^{-} \cap A^{(1)}} \pair{\ve_{k}}{\gamma_{a}^{\vee}} \\ 
&= \bigpair{\ve_{k}}{\sum_{a \in A^{-} \cap A^{(1)}} \gamma_{a}^{\vee}} \\ 
&= \pair{\ve_{k}}{\down(A^{(1)})}. 
\end{align*}
In addition, we have $\wt(A) = \ed(A^{(1)})\ve_{k}$. 
Also, since all the roots in $\Gamma_{k}^{\ast}(k)$ are positive roots, while those in $\Theta_{k}$ are negative roots, we see that $n(A) = |A^{(2)}|$. 
Therefore, we see that
\begin{align}
\begin{split}
& \eqref{eq:C-IC_1} \\
&= \sum_{B \in \CA(w, \Gamma_{k}(k))} (-1)^{|B|} q^{-\pair{\lambda + \ve_{k}}{\down(B)}} \\ 
& \quad \times \frac{1}{1-q^{-\pair{\lambda+\vpi_{k}}{\alpha_{k}^{\vee}}}} \sum_{A \in \CA(\ed(B), \Gamma_{k-1, k})} (-1)^{n(A)} q^{-\pair{\ve_{k}}{\down(A^{(1)})}} e^{\ed(A^{(1)}) \ve_{k}} \gch V_{\ed(A)t_{\down(A)}}^{-}(\lambda) 
\end{split} \nonumber \\ 
\begin{split}
&= \frac{1}{1-q^{-\pair{\lambda+\vpi_{k}}{\alpha_{k}^{\vee}}}} \sum_{B \in \CA(w, \Gamma_{k}(k))} \sum_{A \in \CA(\ed(B), \Gamma_{k-1, k})} (-1)^{|B|} (-1)^{|A^{(2)}|} \\ 
& \hspace*{40mm} \times q^{-\pair{\lambda + \ve_{k}}{\down(B)}} e^{\ed(A^{(1)})\ve_{k}} q^{-\pair{\lambda+\ve_{k}}{\down(A^{(1)})}} \gch V_{\ed(A)t_{\down(A^{(2)})}}^{-}(\lambda)
\end{split} \nonumber \\ 
\begin{split}
&= \frac{1}{1-q^{-\pair{\lambda+\vpi_{k}}{\alpha_{k}^{\vee}}}} \sum_{B \in \CA(w, \Gamma_{k}(k))} \sum_{A \in \CA(\ed(B), \Gamma_{k-1, k})} (-1)^{|B|}(-1)^{|A^{(2)}|} \\ 
& \hspace*{40mm} \times q^{-\pair{\lambda+\ve_{k}}{\down(B)+\down(A^{(1)})}} e^{\ed(A^{(1)})\ve_{k}} \gch V_{\ed(A^{(2)})t_{\down(A^{(2)})}}^{-}(\lambda) 
\end{split} \nonumber \\ 
\begin{split}
&= \frac{1}{1-q^{-\pair{\lambda+\vpi_{k}}{\alpha_{k}^{\vee}}}} \sum_{B \in \CA(w, \Gamma_{k}(k))} \sum_{A^{(1)} \in \CA(\ed(B), \Gamma_{k}^{\ast}(k))} (-1)^{|B|} q^{-\pair{\lambda+\ve_{k}}{\down(B)+\down(A^{(1)})}} \\ 
& \hspace*{40mm} \times e^{\ed(A^{(1)})\ve_{k}} \sum_{A^{(2)} \in \CA(\ed(A^{(1)}), \Theta_{k})} (-1)^{|A^{(2)}|} \gch V_{\ed(A^{(2)})t_{\down(A^{(2)})}}^{-}(\lambda). 
\end{split} \label{eq:C-IC_2}
\end{align}

Now, we define an involution on the set 
\begin{equation*}
\CP := \{ (B, A^{(1)}) \mid B \in \CA(w, \Gamma_{k}(k)), \ A^{(1)} \in \CA(\ed(B), \Gamma_{k}^{\ast}(k)) \}. 
\end{equation*}
There are the following six cases: 
\begin{enu}
\item {[$\max B \prec \min A^{(1)}$]} or {[$B \not= \emptyset$ and $A^{(1)} = \emptyset$]}, 
\item {[$\max B \succ \min A^{(1)}$]} or {[$B = \emptyset$ and $A^{(1)} \not= \emptyset$]}, 
\item $\max B = \min A^{(1)} = (k, k+1)$, and one of the following holds: 
\begin{itemize}
\item $\max (B \setminus \{ (k, k+1) \}) \prec \min (A^{(1)} \setminus \{ (k, k+1) \})$ or 
\item $B \setminus \{ (k, k+1) \} \not= \emptyset$ and $A^{(1)} = \{ (k, k+1) \}$,
\end{itemize} 
\item $\max B = \min A^{(1)} = (k, k+1)$, and one of the following holds: 
\begin{itemize}
\item $\max (B \setminus \{ (k, k+1) \}) \succ \min (A^{(1)} \setminus \{ (k, k+1) \})$ or 
\item $B = \{ (k, k+1) \}$ and $A^{(1)} \setminus \{ (k, k+1) \} \not= \emptyset$, 
\end{itemize}
\item $B = A^{(1)} = \{ (k, k+1) \}$, 
\item $B = A^{(1)} = \emptyset$. 
\end{enu}
Here we remark that, if we have a directed path $v \xrightarrow{\alpha} vs_{\alpha} \xrightarrow{\alpha} v$ in $\QBG(W)$ for $v \in W$ and $\alpha \in \Delta^{+}$, then $\alpha$ must be a simple root. Conversely, for $v \in W$ and a simple root $\alpha$, $v \xrightarrow{\alpha} vs_{\alpha} \xrightarrow{\alpha} v$ is a directed path in $\QBG(W)$. 
For $(B, A^{(1)}) \in \CP$, we define $\iota(B, A^{(1)}) = (B', A'^{(1)}) \in \CP$ as follows: 
\begin{itemize}
\item if $(B, A^{(1)})$ satisfies (1) above, then set 
\begin{equation*}
B' := B \setminus \{ \max B \}, \quad A'^{(1)} := A^{(1)} \sqcup \{ \max B \}; 
\end{equation*}
\item if $(B, A^{(1)})$ satisfies (2) above, then set 
\begin{equation*}
B' := B \sqcup \{ \min A^{(1)} \}, \quad A'^{(1)} := A^{(1)} \setminus \{ \min A^{(1)} \}; 
\end{equation*}
\item if $(B, A^{(1)})$ satisfies (3) above, then set 
\begin{equation*}
B' := B \setminus \{ \max (B \setminus \{ (k, k+1) \}) \}, \quad A'^{(1)} := A' \sqcup \{ \max (B \setminus \{ (k, k+1) \}; 
\end{equation*}
\item if $(B, A^{(1)})$ satisfies (4) above, then set 
\begin{equation*}
B' := B \sqcup \{ \min (A'^{(1)} \setminus \{ (k, k+1) \}, \quad A'^{(1)} := A' \setminus \{ \min (A'^{(1)} \setminus \{ (k, k+1) \}) \}; 
\end{equation*}
\item if $(B, A^{(1)})$ satisfies (5) or (6) above, then set 
\begin{equation*}
B' := B, \quad A'^{(1)} := A^{(1)}. 
\end{equation*}
\end{itemize}
It is clear that $\iota$ defines an involution on $\CP$. 
Moreover, in cases (1) and (3) (resp., (2) and (4)), we have 
\begin{itemize}
\item $|B'| = |B| - 1$ (resp., $|B'| = |B| + 1$), 
\item $\down(B') + \down(A'^{(1)}) = \down(B) + \down(A^{(1)})$, and 
\item $\ed(A'^{(1)}) = \ed(A^{(1)})$. 
\end{itemize}
This implies that 
\begin{align*}
& \sum_{\substack{(B, A^{(1)}) \in \CP \\ \text{$(B, A^{(1)})$ satisfies one of (1)--(4)}}} (-1)^{|B|} q^{-\pair{\lambda+\ve_{k}}{\down(B)+\down(A^{(1)})}} e^{\ed(A^{(1)})\ve_{k}} \\ 
& \hspace*{40mm} \times \sum_{A^{(2)} \in \CA(\ed(A^{(1)}), \Theta_{k})} (-1)^{|A^{(2)}|} \gch V_{\ed(A^{(2)})t_{\down(A^{(2)})}}^{-}(\lambda) = 0. 
\end{align*}
Therefore, we conclude that 
\begin{align*}
\begin{split}
\eqref{eq:C-IC_2} &= \frac{1}{1-q^{-\pair{\lambda+\vpi_{k}}{\alpha_{k}^{\vee}}}} \left( \underbrace{-q^{-\pair{\lambda+\ve_{k}}{\alpha_{k}^{\vee}}} e^{w\ve_{k}} \sum_{A^{(2)} \in \CA(w, \Theta_{k})} (-1)^{|A^{(2)}|} \gch V_{\ed(A^{(2)})t_{\down(A^{(2)})}}^{-}(\lambda)}_{B = A^{(1)} = \{ (k, k+1) \}} \right. \\ 
& \hspace*{60mm} \left. + \underbrace{e^{w\ve_{k}} \sum_{A^{(2)} \in \CA(w, \Theta_{k})} (-1)^{|A^{(2)}|} \gch V_{\ed(A^{(2)})t_{\down(A^{(2)})}}^{-}(\lambda)}_{B = A^{(1)} = \emptyset} \right) 
\end{split} \displaybreak[1] \\ 
&= \frac{1-q^{-\pair{\lambda+\ve_{k}}{\alpha_{k}^{\vee}}}}{1-q^{-\pair{\lambda+\vpi_{k}}{\alpha_{k}^{\vee}}}} e^{w\ve_{k}} \sum_{A^{(2)} \in \CA(w, \Theta_{k})} (-1)^{|A^{(2)}|} \gch V_{\ed(A^{(2)})t_{\down(A^{(2)})}}^{-}(\lambda) \\ 
&= e^{w\ve_{k}} \sum_{A^{(2)} \in \CA(w, \Theta_{k})} (-1)^{|A^{(2)}|} \gch V_{\ed(A^{(2)})t_{\down(A^{(2)})}}^{-}(\lambda) \\ 
&= (\text{RHS of \eqref{eq:identity_C-IC}}), 
\end{align*}
as desired; for the third equality, we have used that $\pair{\ve_{k}}{\alpha_{k}^{\vee}} = \pair{\vpi_{k}}{\alpha_{k}^{\vee}} = 1$. This completes the proof of the proposition. 
\end{proof}

\begin{proof}[Proof of Theorem~\ref{thm:IC_first-half}]
We prove the assertion of the theorem by induction on $m = 1, \ldots, n$. 
If $m = 1$, then the assertion immediately follows from Proposition~\ref{prop:key_first-half} since $\CA(w, \Theta_{1}) = \{ \emptyset \}$. 
Let $1 < l \le n$, and assume the assertion for $m = 1, \ldots, l-1$. We will prove the assertion for $m = l$. 
Note that for $A \in \CA(w, \Theta_{l}) \setminus \{ \emptyset \}$, if the index $k$ satisfies $\ed(A)^{-1} w\ve_{l} = \ve_{k}$, then we have $1 \le k \le l-1$. Therefore, by Proposition~\ref{prop:key_first-half}, we see that 
\begin{align*}
& e^{w\ve_{l}} \gch V_{w}^{-}(\lambda) \\ 
\begin{split}
&= \sum_{B \in \CA(w, \Gamma_{l}(l))} (-1)^{|B|} \gch V_{\ed(B)t_{\down(B)}}^{-}(\lambda + \ve_{l}) \\ 
& \quad - \sum_{A \in \CA(w, \Theta_{l}) \setminus \{ \emptyset \}} (-1)^{|A|} e^{w\ve_{l}} \gch V_{\ed(A)t_{\down(A)}}^{-}(\lambda) 
\end{split} \\
\begin{split}
&= \sum_{B \in \CA(w, \Gamma_{l}(l))} (-1)^{|B|} \gch V_{\ed(B)t_{\down(B)}}^{-}(\lambda + \ve_{l}) \\ 
& \quad - \sum_{A \in \CA(w, \Theta_{l}) \setminus \{ \emptyset \}} (-1)^{|A|} q^{-\pair{\lambda}{\down(A)}} e^{w\ve_{l}} \gch V_{\ed(A)}^{-}(\lambda) 
\end{split} \\ 
\begin{split}
&= \sum_{B \in \CA(w, \Gamma_{l}(l))} (-1)^{|B|} \gch V_{\ed(B)t_{\down(B)}}^{-}(\lambda + \ve_{l}) \\ 
& \quad + \sum_{k = 1}^{l-1} \sum_{\substack{A \in \CA(w, \Theta_{l}) \setminus \{ \emptyset \} \\ \ed(A)^{-1}w\ve_{l} = \ve_{k}}} (-1)^{|A|-1} q^{-\pair{\lambda}{\down(A)}} \underbrace{e^{w\ve_{l}} \gch V_{\ed(A)}^{-}(\lambda)}_{\text{induction hypothesis}} 
\end{split} \\ 
\begin{split}
&= \sum_{B \in \CA(w, \Gamma_{l}(l))} (-1)^{|B|} \gch V_{\ed(B)t_{\down(B)}}^{-}(\lambda + \ve_{l}) \\ 
& \quad + \sum_{k = 1}^{l-1} \sum_{A \in \CA_{w}^{l, k}} (-1)^{|A|-1} q^{-\pair{\lambda}{\down(A)}} \\ 
& \quad \times \left( \sum_{B \in \CA(\ed(A), \Gamma_{k}(k))} (-1)^{|B|} \gch V_{\ed(B)t_{\down(B)}}^{-}(\lambda + \ve_{k}) \right. \\ 
& \quad \quad \quad + \sum_{j = 1}^{k-1} \sum_{(j_{1}, \ldots, j_{r}) \in \CS_{k, j}} \sum_{A_{1} \in \CA_{\ed(A)}^{k, j_{1}}} \cdots \sum_{A_{r} \in \CA_{\ed(A_{r-1})}^{j_{r-1}, j_{r}}} (-1)^{|A_{1}|+\cdots + |A_{r}|-r} q^{\pair{\ve_{j}}{\down(A_{1}, \ldots, A_{r})}} \\ 
& \left. \quad \quad \quad \times \sum_{B \in \CA(\ed(A_{r}), \Gamma_{j}(j))} (-1)^{|B|} \gch V_{\ed(B)t_{\down(A_{1}, \ldots, A_{r}, B)}}^{-}(\lambda+\ve_{j}) \right)
\end{split} \displaybreak[1] \\ 
\begin{split}
&= \sum_{B \in \CA(w, \Gamma_{l}(l))} (-1)^{|B|} \gch V_{\ed(B)t_{\down(B)}}^{-}(\lambda + \ve_{l}) \\ 
& \quad + \sum_{k = 1}^{l-1} \sum_{A \in \CA_{w}^{l, k}} (-1)^{|A|-1} \\ 
& \quad \times \left( \sum_{B \in \CA(\ed(A), \Gamma_{k}(k))} (-1)^{|B|} q^{\pair{\ve_{k}}{\down(A)}} \gch V_{\ed(B)t_{\down(B)+\down(A)}}^{-}(\lambda + \ve_{k}) \right. \\ 
& \quad \quad \quad + \sum_{j = 1}^{k-1} \sum_{(j_{1}, \ldots, j_{r}) \in \CS_{k, j}} \sum_{A_{1} \in \CA_{\ed(A)}^{k, j_{1}}} \cdots \sum_{A_{r} \in \CA_{\ed(A_{r-1})}^{j_{r-1}, j_{r}}} (-1)^{|A_{1}|+\cdots + |A_{r}|-r} q^{\pair{\ve_{j}}{\down(A_{1}, \ldots, A_{r})}} \\ 
& \left. \quad \quad \quad \times \sum_{B \in \CA(\ed(A_{r}), \Gamma_{j}(j))} (-1)^{|B|} q^{\pair{\ve_{j}}{\down(A)}} \gch V_{\ed(B)t_{\down(A, A_{1}, \ldots, A_{r}, B)}}^{-}(\lambda+\ve_{j}) \right)
\end{split} \\ 
\begin{split}
&= \sum_{B \in \CA(w, \Gamma_{l}(l))} (-1)^{|B|} \gch V_{\ed(B)t_{\down(B)}}^{-}(\lambda + \ve_{l}) \\ 
& \quad + \sum_{k = 1}^{l-1} \sum_{(j_{1}, \ldots, j_{r}) \in \CS_{l, j}} \sum_{A_{1} \in \CA_{w}^{l, j_{1}}} \cdots \sum_{A_{r} \in \CA_{\ed(A_{r-1})}^{j_{r-1}, j_{r}}} (-1)^{|A_{1}|+\cdots+|A_{r}|-r} q^{\pair{\ve_{j}}{\down(A_{1}, \ldots, A_{r})}} \\ 
& \quad \times \sum_{B \in \CA(\ed(A_{r}), \Gamma_{j}(j))} (-1)^{|B|} \gch V_{\ed(B)t_{\down(A_{1}, \ldots, A_{r}, B)}}^{-}(\lambda + \ve_{j}), 
\end{split}
\end{align*}
as desired. Thus, the assertion also holds for $m = l$. 
This proves the theorem. 
\end{proof}

\subsection{Second-half case}
The following proposition is a key to the proof of the second half of our identities of inverse Chevalley type. 
\begin{prop} \label{prop:key_second-half}
Let $w \in W$, $\lambda \in P^{+}$ such that $\lambda - \ve_{k} \in P^{+}$. Then we have 
\begin{equation*}
\begin{split}
& \sum_{B \in \CA(w, \Theta_{k})} (-1)^{|B|} \gch V_{\ed(B)t_{\down(B)}}^{-}(\lambda - \ve_{k}) \\ 
& = \sum_{A \in \CA(w, \Gamma_{k}(k))} (-1)^{|A|} e^{-w\ve_{k}} \gch V_{\ed(A)t_{\down(A)}}^{-}(\lambda). 
\end{split}
\end{equation*}
\end{prop}
\begin{proof}
By replacing $\lambda$ in equation \eqref{eq:identity_C-IC} with $\lambda - \ve_{k}$ and multiplying both sides of equation \eqref{eq:identity_C-IC} by $e^{-w\ve_{k}}$, we obtain the desired identity. 
\end{proof}

\begin{proof}[Proof of Theorem~\ref{thm:IC_second-half}]
We prove the assertion of the theorem by downward induction on $m = n, n-1, \ldots, 1$. First, assume that $m = n$. Then, by Proposition~\ref{prop:key_second-half} and Corollary~\ref{cor:IC_first-half_complete}, we see that 
\begin{align*}
& e^{-w\ve_{m}} \gch V_{w}^{-}(\lambda) \\ 
\begin{split}
&= \sum_{B \in \CA(w, \Theta_{n})} (-1)^{|B|} \gch V_{\ed(B)t_{\down(B)}}^{-}(\lambda - \ve_{n}) \\ 
& \quad - \sum_{A \in \CA(w, \Gamma_{n}(n)) \setminus \{ \emptyset \}} (-1)^{|A|} e^{-w\ve_{n}} \gch V_{\ed(A)t_{\down(A)}}^{-}(\lambda)
\end{split} \\ 
\begin{split}
&= \sum_{B \in \CA(w, \Theta_{n})} (-1)^{|B|} \gch V_{\ed(B)t_{\down(B)}}^{-}(\lambda - \ve_{n}) \\ 
& \quad + \sum_{k = 1}^{n} \sum_{A \in \CA_{w}^{\overline{n}, k}} (-1)^{|A|-1} e^{-w\ve_{n}} \gch V_{\ed(A)t_{\down(A)}}^{-}(\lambda)
\end{split} \displaybreak[1] \\ 
\begin{split}
&= \sum_{B \in \CA(w, \Theta_{n})} (-1)^{|B|} \gch V_{\ed(B)t_{\down(B)}}^{-}(\lambda - \ve_{n}) \\ 
& \quad + \sum_{k = 1}^{n} \sum_{A \in \CA_{w}^{\overline{n}, k}} (-1)^{|A|-1} e^{\ed(A)\ve_{k}} \gch V_{\ed(A)t_{\down(A)}}^{-}(\lambda)
\end{split} \\ 
\begin{split}
&= \sum_{B \in \CA(w, \Theta_{n})} (-1)^{|B|} \gch V_{\ed(B)t_{\down(B)}}^{-}(\lambda - \ve_{n}) \\ 
& \quad + \sum_{k = 1}^{n} \sum_{A \in \CA_{w}^{\overline{n}, k}} (-1)^{|A|-1} \left( q^{\pair{\ve_{k}}{\down(A)}} \sum_{B \in \CA(\ed(A), \Gamma_{k}(k))} (-1)^{|B|} \gch V_{\ed(B)t_{\down(B)}}^{-}(\lambda + \ve_{k}) \right. \\ 
& \quad + \sum_{j = 1}^{k-1} \sum_{(j_{1}, \ldots, j_{r}) \in \CS_{k, j}} \sum_{A_{1} \in \CA_{\ed(A)}^{k, j_{1}}} \cdots \sum_{A_{r} \in \CA_{\ed(A_{r-1})}^{j_{r-1}, j_{r}}} (-1)^{|A_{1}|+\cdots+|A_{r}|-r} q^{\pair{\ve_{j}}{\down(A, A_{1}, \ldots, A_{r})}} \\ 
& \hspace*{30mm} \left. \times \sum_{B \in \CA(\ed(A_{r}), \Gamma_{j}(j))} (-1)^{|B|} \gch V_{\ed(B)t_{\down(A, A_{1}, \ldots, A_{r}, B)}}^{-}(\lambda + \ve_{j}) \right) 
\end{split} \\ 
\begin{split}
&= \sum_{B \in \CA(w, \Theta_{n})} (-1)^{|B|} \gch V_{\ed(B)t_{\down(B)}}^{-}(\lambda - \ve_{n}) \\ 
& \quad + \sum_{j = 1}^{n} \sum_{(j_{1}, \ldots, j_{r}) \in \CS_{\overline{n}, j}} \sum_{A_{1} \in \CA_{w}^{\overline{n}, j_{1}}} \cdots \sum_{A_{r} \in \CA_{\ed(A_{r-1})}^{j_{r-1}, j_{r}}} (-1)^{|A_{1}|+\cdots+|A_{r}|-r} q^{\pair{\ve_{j}}{\down(A_{1}, \ldots, A_{r})}} \\ 
& \hspace*{30mm} \times \sum_{B \in \CA(\ed(A_{r}), \Gamma_{j}(j))} (-1)^{|B|} \gch V_{\ed(B)t_{\down(A_{1}, \ldots, A_{r}, B)}}^{-}(\lambda + \ve_{j}), 
\end{split} 
\end{align*}
as desired. 
Let $1 \le l < n$, and assume the assertion for $m = n, n-1, \ldots, l+1$. We prove the assertion for $m = l$. For $A \in \CA(w, \Gamma_{l}(l)) \setminus \{ \emptyset \}$, if the index $k$ satisfies $\ed(A)^{-1}w(-\ve_{l}) = \ve_{k}$, then we have $k = 1, \ldots, n-1, n, \overline{n}, \overline{n-1}, \ldots, \overline{l+1}$. Therefore, by Proposition~\ref{prop:key_second-half} (and Proposition~\ref{prop:gch_translation}), we compute 
\begin{align*}
& e^{-w\ve_{l}} \gch V_{w}^{-}(\lambda) \\ 
\begin{split}
&= \sum_{B \in \CA(w, \Theta_{l})} (-1)^{|B|} \gch V_{\ed(B)t_{\down(B)}}^{-}(\lambda - \ve_{l}) \\ 
& \quad - \sum_{A \in \CA(w, \Gamma_{l}(l)) \setminus \{ \emptyset \}} (-1)^{|A|} e^{-w\ve_{l}} \gch V_{\ed(A)t_{\down(A)}}^{-}(\lambda) 
\end{split} \\ 
\begin{split}
&= \sum_{B \in \CA(w, \Theta_{l})} (-1)^{|B|} \gch V_{\ed(B)t_{\down(B)}}^{-}(\lambda - \ve_{l}) \\ 
& \quad + \sum_{k = l+1}^{n} \sum_{A \in \CA_{w}^{\overline{l}, \overline{k}}} (-1)^{|A|-1} e^{-w\ve_{l}} \gch V_{\ed(A)t_{\down(A)}}^{-}(\lambda) \\ 
& \quad + \sum_{k = 1}^{n} \sum_{A \in \CA_{w}^{\overline{l}, k}} (-1)^{|A|-1} e^{-w\ve_{l}} \gch V_{\ed(A)t_{\down(A)}}^{-}(\lambda) 
\end{split} \displaybreak[1] \\ 
\begin{split}
&= \sum_{B \in \CA(w, \Theta_{l})} (-1)^{|B|} \gch V_{\ed(B)t_{\down(B)}}^{-}(\lambda - \ve_{l}) \\ 
& \quad + \sum_{k = l+1}^{n} \sum_{A \in \CA_{w}^{\overline{l}, \overline{k}}} (-1)^{|A|-1} q^{-\pair{\lambda}{\down(A)}} \underbrace{e^{-\ed(A)\ve_{k}} \gch V_{\ed(A)}^{-}(\lambda)}_{\text{induction hypothesis}} \\ 
& \quad + \sum_{k = 1}^{n} \sum_{A \in \CA_{w}^{\overline{l}, k}} (-1)^{|A|-1} \underbrace{e^{\ed(A)\ve_{k}} \gch V_{\ed(A)t_{\down(A)}}^{-}(\lambda)}_{\text{Corollary~\ref{cor:IC_first-half_complete}}} 
\end{split} \\ 
\begin{split}
&= \sum_{B \in \CA(w, \Theta_{l})} (-1)^{|B|} \gch V_{\ed(B)t_{\down(B)}}^{-}(\lambda - \ve_{l}) \\ 
& \quad + \sum_{k = l+1}^{n} \sum_{A \in \CA_{w}^{\overline{l}, \overline{k}}} (-1)^{|A|-1} q^{-\pair{\lambda}{\down(A)}} \left( \sum_{B \in \CA(\ed(A), \Theta_{k})} (-1)^{|B|} \gch V_{\ed(B)t_{\down(B)}}^{-}(\lambda - \ve_{k}) \right. \\ 
& \quad + \sum_{j = k+1}^{n} \sum_{(j_{1}, \ldots, j_{r}) \in \CS_{\overline{k}, \overline{j}}} \sum_{A_{1} \in \CA_{\ed(A)}^{\overline{k}, j_{1}}} \cdots \sum_{A \in \CA_{\ed(A_{r-1})}^{j_{r-1}, j_{r}}} (-1)^{|A_{1}|+\cdots+|A_{r}|-r} q^{-\pair{\ve_{j}}{\down(A_{1}, \ldots, A_{r})}} \\ 
& \quad \times \sum_{B \in \CA(\ed(A_{r}), \Theta_{j})} (-1)^{|B|} \gch V_{\ed(B)t_{\down(A_{1}, \ldots, A_{r}, B)}}^{-}(\lambda - \ve_{j}) \\ 
& \quad + \sum_{j = 1}^{n} \sum_{(j_{1}, \ldots, j_{r}) \in \CS_{\overline{k}, j}} \sum_{A_{1} \in \CA_{\ed(A)}^{\overline{k}, j_{1}}} \cdots \sum_{A_{r} \in \CA_{\ed(A_{r-1})}^{j_{r-1}, j_{r}}} (-1)^{|A_{1}|+\cdots+|A_{r}|-r} q^{\pair{\ve_{j}}{\down(A_{1}, \ldots, A_{r})}} \\ 
& \quad \left. \times \sum_{B \in \CA(\ed(A_{r}), \Gamma_{j}(j))} (-1)^{|B|} \gch V_{\ed(B)t_{\down(A_{1}, \ldots, A_{r}, B)}}^{-}(\lambda + \ve_{j}) \right) \\ 
& \quad + \sum_{k = 1}^{n} \sum_{A \in \CA_{w}^{\overline{l}, k}} (-1)^{|A|-1} \left( q^{\pair{\ve_{k}}{\down(A)}} \sum_{B \in \CA(\ed(A), \Gamma_{k}(k))} (-1)^{|B|} \gch V_{\ed(B)t_{\down(B)+\down(A)}}^{-}(\lambda + \ve_{k}) \right. \\ 
& \quad + \sum_{j = 1}^{k-1} \sum_{(j_{1}, \ldots, j_{r}) \in \CS_{k, j}} \sum_{A_{1} \in \CA_{\ed(A)}^{k, j_{1}}} \cdots \sum_{A_{r} \in \CA_{\ed(A_{r-1})}^{j_{r-1}, j_{r}}} (-1)^{|A_{1}|+\cdots+|A_{r}|-r} q^{\pair{\ve_{j}}{\down(A, A_{1}, \ldots, A_{r})}} \\ 
& \quad \left. \times \sum_{B \in \CA(\ed(A_{r}), \Gamma_{j}(j))} (-1)^{|B|} \gch V_{\ed(B)t_{\down(A, A_{1}, \ldots, A_{r}, B)}}^{-}(\lambda + \ve_{j}) \right) 
\end{split} \\ 
\begin{split}
&= \sum_{B \in \CA(w, \Theta_{l})} (-1)^{|B|} \gch V_{\ed(B)t_{\down(B)}}^{-}(\lambda - \ve_{l}) \\ 
& \quad + \sum_{j = l+1}^{n} \sum_{(j_{1}, \ldots, j_{r}) \in \CS_{\overline{l}, \overline{j}}} \sum_{A_{1} \in \CA_{w}^{\overline{l}, j_{1}}} \cdots \sum_{A_{r} \in \CA_{\ed(A_{r-1})}^{j_{r-1}, j_{r}}} (-1)^{|A_{1}|+\cdots+|A_{r}|-r} q^{-\pair{\ve_{j}}{\down(A_{1}, \ldots, A_{r})}} \\ 
& \quad \times \sum_{B \in \CA(\ed(A_{r}), \Theta_{j})} (-1)^{|B|} \gch V_{\ed(B)t_{\down(A_{1}, \ldots, A_{r}, B)}}^{-}(\lambda - \ve_{j}) \\ 
& \quad + \sum_{j = 1}^{n} \sum_{(j_{1}, \ldots, j_{r}) \in \CS_{\overline{l}, j}} \sum_{A_{1} \in \CA_{w}^{\overline{l}, j_{1}}} \cdots \sum_{A_{r-1} \in \CA_{\ed(A_{r-1})}^{j_{r-1}, j_{r}}} (-1)^{|A_{1}|+\cdots+|A_{r}|-r} q^{\pair{\ve_{j}}{\down(A_{1}, \ldots, A_{r})}} \\ 
& \quad \times \sum_{B \in \CA(\ed(A_{r}), \Gamma_{j}(j))} (-1)^{|B|} \gch V_{\ed(B)t_{\down(A_{1}, \ldots, A_{r}, B)}}^{-}(\lambda + \ve_{j}), 
\end{split}
\end{align*}
as desired. By downward induction, this completes the proof of the theorem. 
\end{proof}

\section{Proof of Theorem~\ref{thm:IC_cancellation-free_first-half}} \label{sec:cancellation-free_proof}
We will derive the cancellation-free form of the first-half identities of inverse Chevalley type (Theorem~\ref{thm:IC_cancellation-free_first-half}). 
For this purpose, we need the following lemmas on edges of the quantum Bruhat graph. 
We continue to assume that $\Fg$ is of type $C_{n}$. 
Recall that a total order $<$ on $[\overline{n}]$ is defined by: $1 < 2 < \cdots < n < \overline{n} < \overline{n-1} < \cdots < \overline{1}$; for each $1 \le k \le n$, we define an order $\prec_{k}$ (resp., $\prec_{\overline{k}}$) on $[\overline{n}]$ by: $k \prec_{k} k+1 \prec_{k} \cdots \prec_{k} n \prec_{k} \overline{n} \prec_{k} \overline{n-1} \prec_{k} \cdots \prec_{k} \overline{1} \prec_{k} 1 \prec_{k} 2 \prec_{k} \cdots \prec_{k} k-1$ (resp., $\overline{k} \prec_{\overline{k}} \overline{k-1} \prec_{\overline{k}} \cdots \prec_{\overline{k}} \overline{1} \prec_{\overline{k}} 1 \prec_{\overline{k}} 2 \prec_{\overline{k}} \cdots \prec_{\overline{k}} n \prec_{\overline{k}} \overline{n} \prec_{\overline{k}} \overline{n-1} \prec_{\overline{k}} \cdots \prec_{\overline{k}} \overline{k+1}$). For $a_{1}, \ldots, a_{r} \in [\overline{n}]$ with $r \ge 2$, we write $a_{1} \prec a_{2} \prec \cdots \prec a_{r}$ if $a_{1} \prec_{a_{1}} a_{2} \prec_{a_{1}} \cdots \prec_{a_{1}} a_{r}$ (the order $\prec$ is different from that introduced in the proof of Proposition~\ref{prop:key_first-half}). 
Also, on the set $[\overline{n}]$, we define the sign function $\sgn (\cdot)$: for $a \in [\overline{n}]$, we set 
\begin{equation*}
\sgn(a) := \begin{cases} 1 & \text{if $a = 1, 2, \ldots, n$, } \\ -1 & \text{if $a = \overline{n}, \overline{n-1}, \ldots, \overline{1}$. } \end{cases}
\end{equation*}
We know the following useful criterion. 

\begin{lem} [{\cite[Proposition~5.7]{L}}] \label{lem:QBG_criterion}
Let $w \in W$. 
\begin{enu}
\item Let $1 \le k < l \le n$. Then, $w \xrightarrow{(k, l)} ws_{(k, l)}$ is an edge in $\QBG(W)$ if and only if there does not exist $k < j < l$ such that $w(k) \prec w(j) \prec w(l)$. 
\item Let $1 \le k < l \le n$. Then, $w \xrightarrow{(k, \overline{l})} ws_{(k, \overline{l})}$ is an edge in $\QBG(W)$ if and only if the following hold: 
\begin{itemize}
\item $w(k) < w(\overline{l})$; 
\item $\sgn(w(k)) = \sgn(w(l))$; and 
\item there does not exist $k < j < \overline{l}$ such that $w(k) < w(j) < w(\overline{l})$. 
\end{itemize}
\item Let $1 \le k \le n$. Then, $w \xrightarrow{(k, \overline{k})} ws_{(k, \overline{k})}$ is an edge in $\QBG(W)$ if and only if there does not exist $k < j < \overline{k}$ such that $w(k) \prec w(j) \prec w(\overline{k})$. 
\end{enu}
\end{lem}

By using this criterion, we can show the following three lemmas. 

\begin{lem} \label{lem:QBG_exchange}
Let $w \in W$, and $1 \le k < l < m \le n$. Then, the following are equivalent: 
\begin{enu}
\item $w \xrightarrow{(k, m)} ws_{(k, m)}$ and $w \xrightarrow{(l, m)} ws_{(l, m)} \xrightarrow{(k, l)} ws_{(l, m)}s_{(k, l)}$; 
\item $w \xrightarrow{(k, m)} ws_{(k, m)}$ and $w \xrightarrow{(l, m)} ws_{(l, m)}$; 
\item $w \xrightarrow{(k, m)} ws_{(k, m)} \xrightarrow{(l, m)} ws_{(k, m)}s_{(l, m)}$. 
\end{enu}
\end{lem}

\begin{lem} \label{lem:QBG_exchange2}
Let $w \in W$. Take $1 \le k_{1} < l_{1} \le n$ and $1 \le k_{2} < l_{2} \le n$ such that $\{ k_{1}, l_{1} \} \cap \{ k_{2}, l_{2} \} = \emptyset$. 
Then, the following are equivalent: 
\begin{enu}
\item we have the directed path $w \xrightarrow{(k_{1}, l_{1})} ws_{(k_{1}, l_{1})} \xrightarrow{(k_{2}, l_{2})} ws_{(k_{1}, l_{1})}s_{(k_{2}, l_{2})}$; 
\item we have the directed path $w \xrightarrow{(k_{2}, l_{2})} ws_{(k_{2}, l_{2})} \xrightarrow{(k_{1}, l_{1})} ws_{(k_{2}, l_{2})}s_{(k_{1}, l_{1})}$. 
\end{enu}
\end{lem}

\begin{lem} \label{lem:QBG_existence}
Let $w \in W$, $m = 1, \ldots, n$, and take $a_{1}, \ldots, a_{s} \in \{ k \in [1, m-1] \mid w \xrightarrow{(k, m)} ws_{(k, m)} \}$ such that $a_{1} < \cdots < a_{s}$; 
by Lemma~\ref{lem:QBG_exchange}, we have the directed path 
\begin{equation*}
w = y_{0} \xrightarrow{(a_{1}, m)} y_{1} \xrightarrow{(a_{2}, m)} \cdots \xrightarrow{(a_{s}, m)} y_{s}
\end{equation*}
in $\QBG(W)$. Let us take $c < a_{1}$ such that 
\begin{itemize}
\item $y_{s} \xrightarrow{(c, a_{1})} y_{s}s_{(c, a_{1})} =: z$ is an edge in $\QBG(W)$, and 
\item $w \xrightarrow{(c, m)} ws_{(c, m)}$ is an edge in $\QBG(W)$. 
\end{itemize}
For $p < a_{1}$, if $w \xrightarrow{(p, m)} ws_{(p, m)}$ is an edge in $\QBG(W)$, then we have $p < c$. 
\end{lem}

\begin{cor} \label{cor:IC_cancellation_minimum}
Let $w \in W$, $m = 1, \ldots, n$, and let $\{ a_{1} < \cdots < a_{s} \} = \{ k \in [1, m-1] \mid w \xrightarrow{(k, m)} ws_{(k, m)} \}$ with $s \ge 2$; 
by Lemma~\ref{lem:QBG_existence}, for $2 = b_{1} < \cdots < b_{u} \le s$, we have the directed path 
\begin{equation*}
w = z_{0} \xrightarrow{(a_{b_{1}}, m)} z_{1} \xrightarrow{(a_{b_{2}}, m)} \cdots \xrightarrow{(a_{b_{u}}, m)} z_{u}
\end{equation*}
in $\QBG(W)$. Then, $a_{1}$ is equal to the minimal $c$ such that $1 \le c < a_{b_{1}}$ for which $z_{u} \xrightarrow{(c, a_{b_{1}})} z_{u}s_{(c, a_{b_{1}})}$ is an edge in $\QBG(W)$. 
\end{cor}
\begin{proof}
Let us take the minimal $c$ for which $z_{u} \xrightarrow{(c, a_{b_{1}})} z_{u}s_{(c, a_{b_{1}})}$ is an edge in $\QBG(W)$. 
Note that such a $c$ exists since $z_{u} \xrightarrow{(a_{1}, a_{b_{1}})} z_{u} s_{(a_{1}, a_{b_{1}})}$ is an edge in $\QBG(W)$ by Lemmas~\ref{lem:QBG_exchange} and \ref{lem:QBG_exchange2}. This also implies that $c \le a_{1}$. 
Assume, for a contradiction, that $c < a_{1}$. Then, since $a_{1}$ is the minimum of the set $\{k \in [1, m-1] \mid w \xrightarrow{(k, m)} ws_{(k, m)} \}$, $w \xrightarrow{(c, m)} ws_{(c, m)}$ is not an edge in $\QBG(W)$. Also, we see that $a_{1} < a_{2} \le a_{b_{1}}$ and that $w \xrightarrow{(a_{1}, m)} ws_{(a_{1}, m)}$ is an edge in $\QBG(W)$. 
Therefore, by Lemma~\ref{lem:QBG_existence}, we obtain $a_{1} < c$, which is a contradiction. 
Hence we conclude that $c = a_{1}$, as desired. This proves the corollary.
\end{proof}

Theorem~\ref{thm:IC_cancellation-free_first-half} follows immediately from the following key proposition and Theorem~\ref{thm:IC_first-half}. 
Let $\BZ[q^{-1}][W]$ denote the group algebra of $W$ with coefficients in $\BZ[q^{-1}]$; 
the elements of $\BZ[q^{-1}][W]$ are of the form $\sum_{v \in W} c_{v}(q^{-1}) v$, with $c_{v}(q^{-1}) \in \BZ[q^{-1}]$. 
\begin{prop}
Let $w \in W$, $m = 1, \ldots, n$, and $j = 1, \ldots, m-1$. Then, there holds the following equality in $\BZ[q^{-1}][W]$: 
\begin{equation*}
\begin{split}
& \sum_{(j_{1}, \ldots, j_{r}) \in \CS_{m, j}} \sum_{A_{1} \in \CA_{w}^{m, j_{1}}} \cdots \sum_{A_{r} \in \CA_{\ed(A_{r-1})}^{j_{r-1}, j_{r}}} (-1)^{|A_{1}|+\cdots+|A_{r}|-r} q^{-\pair{\lambda}{\down(A_{1}, \ldots, A_{r})}} \ed(A_{r}) \\ 
&= q^{-\pair{\lambda}{\wt(\bp_{m, j}(w))}} \ed(\bp_{m, j}(w)). 
\end{split}
\end{equation*}
\end{prop}
\begin{proof}
We prove the assertion of the proposition by induction on $m-j$. If $m-j = 1$, then the assertion is clear since $\CS_{m, m-1} = \{ (m-1) \}$ and $\CA_{w}^{m, m-1} = \{ \{ m-1 \} \}$. Assume that $m-j > 1$. By Lemma~\ref{lem:QBG_exchange}, we can verify that if $\{ a_{1} < \cdots < a_{s} \} = \{ k \in [1, m-1] \mid w \xrightarrow{(k, m)} ws_{(k, m)} \}$ (note that $a_{s} = m-1$), then 
\begin{equation*}
\CA_{w}^{m, j} = \begin{cases} \{ \{j, a_{c_{1}}, \ldots, a_{c_{u}}\} \mid l < c_{1} < \cdots < c_{u} \le s \} & \text{if $j = a_{l}$ for some $l = 1, \ldots, s$, } \\ \emptyset & \text{if $j \not= a_{1}, \ldots, a_{s}$. } \end{cases}
\end{equation*}
Therefore, we compute
\begin{align}
& \sum_{(j_{1}, \ldots, j_{r}) \in \CS_{m, j}} \sum_{A_{1} \in \CA_{w}^{m, j_{1}}} \cdots \sum_{A_{r} \in \CA_{\ed(A_{r-1})}^{j_{r-1}, j_{r}}} (-1)^{|A_{1}|+\cdots+|A_{r}|-r} q^{-\pair{\lambda}{\down(A_{1}, \ldots, A_{r})}} \ed(A_{r}) \nonumber \\ 
\begin{split}
&= \sum_{k=1}^{s} \sum_{(j_{1}, \ldots, j_{r}) \in \CS_{a_{k}, j}} \sum_{A \in \CA_{w}^{m, a_{k}}} \sum_{A_{1} \in \CA_{\ed(A)}^{a_{k}, j_{1}}} \cdots \sum_{A_{r} \in \CA_{\ed(A_{r-1})}^{j_{r-1}, j_{r}}} \\ 
& \hspace*{20mm} \times (-1)^{|A|+|A_{1}|+\cdots+|A_{r}|-(r+1)} q^{-\pair{\lambda}{\down(A, A_{1}, \ldots, A_{r})}}\ed(A_{r})
\end{split} \nonumber \displaybreak[1] \\ 
\begin{split}
&= \sum_{k=1}^{s} \sum_{A \in \CA_{w}^{m, a_{k}}} (-1)^{|A|-1} q^{-\pair{\lambda}{\down(A)}} \\ 
& \hspace*{5mm} \times \underbrace{\sum_{(j_{1}, \ldots, j_{r}) \in \CS_{a_{k}, j}} \sum_{A_{1} \in \CA_{\ed(A)}^{a_{k}, j_{1}}} \cdots \sum_{A_{r} \in \CA_{\ed(A_{r-1})}^{j_{r-1}, j_{r}}} (-1)^{|A_{1}|+\cdots+|A_{r}|-r} q^{-\pair{\lambda}{\down(A_{1}, \ldots, A_{r})}}\ed(A_{r})}_{\text{induction hypothesis}}
\end{split} \nonumber \\ 
&= \sum_{k=1}^{s} \sum_{A \in \CA_{w}^{m, a_{k}}} (-1)^{|A|-1} q^{-\pair{\lambda}{\down(A)+\wt(\bp_{a_{k}, j}(\ed(A)))}} \ed(\bp_{a_{k}, j}(\ed(A))). \label{eq:IC_cancellation1}
\end{align}

If $s = 1$, then the assertion is clear since $\CA_{w}^{m, a_{1}} = \{ \{m-1\} \}$. Now, assume that $s \ge 2$. 
For $A \in \CA := (\CA_{w}^{m, a_{1}} \setminus \{ \{a_{1}\} \}) \sqcup (\bigsqcup_{k = 2}^{s} \CA_{w}^{m, a_{k}})$, we define $\iota(A)$ by 
\begin{align*}
A \in \bigsqcup_{k = 2}^{s} \CA_{w}^{m, a_{k}} \quad &\mapsto \quad \iota(A) := A \sqcup \{a_{1}\} \in \CA_{w}^{m, a_{1}} \setminus \{ \{a_{1}\} \}, \\ 
A \in \CA_{w}^{m, a_{1}} \setminus \{ \{a_{1}\} \} \quad &\mapsto \quad \iota(A) := A \setminus \{a_{1}\} \in \CA_{w}^{m, \min (A \setminus \{a_{1}\})} \subset \bigsqcup_{k=2}^{s} \CA_{w}^{m, a_{k}}. 
\end{align*}
We see that this $\iota$ defines an involution on the set $\CA$ such that 
$|\iota(A)| = |A| \pm 1$ for $A \in \CA$. For $A \in \CA_{w}^{m, a_{k}}$ with $k = 2, \ldots, s$, it follows from Corollary~\ref{cor:IC_cancellation_minimum} that the first edge in the directed path $\bp_{a_{k}, j}(\ed(A))$ in $\QBG(W)$ is $\ed(A) \xrightarrow{(a_{1}, a_{k})} \ed(A)s_{(a_{1}, a_{k})} = \ed(\iota(A))$. 
Hence we have $\ed(\bp_{a_{k}, j}(\ed(A))) = \ed(\bp_{a_{1}, j}(\ed(\iota(A))))$. 
Also, the directed path $w \rightarrow \cdots \rightarrow \ed(\iota(A))$ in $\QBG(W)$ corresponding to $\iota(A)$ is a shortest one of length $|\iota(A)| = |A| + 1$, since the order $\prec$ given by $(1, m) \prec \cdots \prec (m-1, m)$ forms a part of a reflection order on the set $\Delta^{+}$ of positive roots. 
Thus, the concatenation $w \rightarrow \cdots \rightarrow \ed(A) \xrightarrow{(a_{1}, a_{k})} \ed(A)s_{(a_{1}, a_{k})}$ of the directed path corresponding to $A$ with the edge $\ed(A) \xrightarrow{(a_{1}, a_{k})} \ed(A)s_{(a_{1}, a_{k})}$ is also a shortest one. 
Here we know (see \cite[Lemma~1\,(2)]{P}) that for any $v, u \in W$, all shortest directed paths from $v$ to $u$ in $\QBG(W)$ have the same weight $\wt(\cdot)$ (see \cite[Lemma~1\,(2)]{P}). It follows that 
\begin{equation*}
\down(A) + \wt(\bp_{a_{k}, j}(\ed(A))) = \down(\iota(A)) + \wt(\bp_{a_{1}, j}(\ed(\iota(A)))). 
\end{equation*}
Therefore, for $A \in \CA_{w}^{m, a_{k}}$ with $k = 2, \ldots, s$, we deduce that 
\begin{equation*}
\begin{split}
& (-1)^{|A|-1} q^{-\pair{\lambda}{\down(A) + \wt(\bp_{a_{k}, j}(\ed(A)))}} \ed(\bp_{a_{k}, j}(\ed(A))) \\ 
& + (-1)^{|\iota(A)|+1} q^{-\pair{\lambda}{\down(\iota(A)) + \wt(\bp_{a_{1}, j}(\ed(\iota(A))))}} \ed(\bp_{a_{1}, j}(\ed(\iota(A)))) = 0. 
\end{split}
\end{equation*}
This implies that 
\begin{align*}
\eqref{eq:IC_cancellation1} = \underbrace{q^{-\pair{\lambda}{\down(\{a_{1}\})+\wt(\bp_{a_{1}, j}(\ed(\{a_{1}\})))}} \ed(\bp_{a_{1}, j}(\ed(\{a_{1}\})))}_{\text{$k = 1$ and $A = \{a_{1}\} \in \CA_{w}^{m, a_{1}}$}} = q^{-\pair{\lambda}{\bp_{m, j}(w)}} \ed(\bp_{m, j}(w)). 
\end{align*}
This proves the proposition. 
\end{proof}

\end{document}